\documentclass[dvips,preprint]{imsart}

\usepackage[OT1]{fontenc}
\usepackage{dsfont}
\usepackage{float}
\usepackage{amsmath,amssymb,amsfonts,amsthm}
\usepackage{url,graphicx,tabularx,array,multirow}
\usepackage[numbers]{natbib}
\usepackage{titlesec}
\usepackage{xypic}
\usepackage{xy}\xyoption{all} \xyoption{poly} \xyoption{knot}
\usepackage[colorlinks,citecolor=blue,urlcolor=blue]{hyperref}

\startlocaldefs
\numberwithin{equation}{section}
\theoremstyle{plain}
\newtheorem{theorem}{Theorem}[section]
\newtheorem{lemma}[theorem]{Lemma}
\newtheorem{proposition}[theorem]{Proposition}
\newtheorem{corollary}[theorem]{Corollary}
\newtheorem{definition}[theorem]{Definition}
\endlocaldefs

\newcommand{\bbR}{\mathbb{R}}

\begin{document}
\bibliographystyle{imsart-number}
\begin{frontmatter}
\title{Bayesian Manifold Regression}
\runtitle{Bayesian Manifold Regression}

\begin{aug}
\author{\fnms{Yun} \snm{Yang}
\thanksref{t2}
\ead[label=e1]{yy84@stat.duke.edu}}
\and
\author{\fnms{David B.} \snm{Dunson}
\thanksref{}
\ead[label=e2]{dunson@stat.duke.edu}}

\thankstext{t2}{Supported by grant ES017436 from the National Institute of Environmental Health Sciences (NIEHS) of the National Institutes of Health
(NIH).}
\runauthor{Y. Yang et al.}

\affiliation{Department of Statistical Science, Duke University\thanksmark{m1}}

\address{Department of Statistical Science\\
Duke University\\
Box 90251\\
NC 27708-0251, Durham, USA\\
\printead{e1}\\
\phantom{E-mail:\ }\printead*{e2}}
\end{aug}

\begin{abstract}
There is increasing interest in the problem of nonparametric regression with high-dimensional predictors.  When the number of predictors $D$ is large, one encounters a daunting problem in attempting to estimate a $D$-dimensional surface based on limited data.  Fortunately, in many applications, the support of the data is concentrated on a $d$-dimensional subspace with $d \ll D$.
Manifold learning attempts to estimate this subspace.  Our focus is on developing computationally tractable and theoretically supported Bayesian nonparametric regression methods in this context.  When the subspace corresponds to a locally-Euclidean compact Riemannian manifold, we show that a Gaussian process regression approach can be applied that leads to the minimax optimal adaptive rate in estimating the regression function under some conditions.  The proposed model bypasses the need to estimate the manifold, and can be implemented using standard algorithms for posterior computation in Gaussian processes.  Finite sample performance is illustrated in an example data analysis.
\end{abstract}



\end{frontmatter}

\section{Introduction}
Dimensionality reduction in nonparametric regression is of increasing interest given the routine collection of high-dimensional predictors.  Our focus is on the regression model
\begin{eqnarray}
Y_i = f( X_i ) + \epsilon_i,\quad \epsilon_i \sim N(0,\sigma^2),\quad i=1,\ldots,n, \label{eq:base}
\end{eqnarray}
where $Y_i \in \bbR$, $X_i \in \bbR^D$, $f$ is an unknown regression function, and $\epsilon_i$ is a residual having variance $\sigma^2$.  We face problems in estimating $f$ accurately due to the moderate to large number of predictors $D$.  Fortunately, in many applications, the predictors have support that is concentrated near a $d$-dimensional subspace $\mathcal{M}$.  If one can learn the mapping from the ambient space to this subspace, the dimensionality of the regression function can be reduced massively from $D$ to $d$, so that $f$ can be much more accurately estimated.

There is an increasingly vast literature on subspace learning, but there remains a lack of approaches that allow flexible non-linear dimensionality reduction, are scalable computationally to moderate to large $D$, have theoretical guarantees and provide a characterization of uncertainty. \cite{Castillo2013} directly constructed a non-stationary Gaussian process prior on a known manifold though rescaling the solutions of the heat equation. However, in many cases, the manifold is not known in advance.

With this motivation, we focus on Bayesian nonparametric regression methods that allow $\mathcal{M}$ to be an unknown Riemannian manifold.  One natural direction is to choose a prior to allow uncertainty in $\mathcal{M}$, while also placing priors on the mapping from $x_i$ to $\mathcal{M}$, the regression function relating the lower-dimensional features to the response, and the residual variance.  Some related attempts have been made in the literature.  \cite{Tokdar2010} propose a logistic Gaussian process model, which allows the conditional response density $f(y|x)$ to be unknown and changing flexibly with $x$, while reducing dimension through projection to a linear subspace.  Their approach is elegant and theoretically grounded, but does not scale efficiently as $D$ increases and is limited by the linear subspace assumption.  Also making the linear subspace assumption, \cite{Reich2011} proposed a Bayesian finite mixture model for sufficient dimension reduction.  \cite{Page2013} instead propose a method for Bayesian nonparametric learning of an affine subspace motivated by classification problems.

There is also a limited literature on Bayesian nonlinear dimensionality reduction.  Gaussian process latent variable models (GP-LVMs) \citep{Lawrence2003} were introduced as a nonlinear alternative to PCA for visualization of high-dimensional data.  \cite{Kundu2011} proposed a related approach that defines separate Gaussian process regression models for the response and each predictor, with these models incorporating shared latent variables to induce dependence.  The latent variables can be viewed as coordinates on a lower dimensional manifold, but daunting problems arise in attempting to learn the number of latent variables, the distribution of the latent variables, and the individual mapping functions while maintaining identifiability restrictions.  \cite{Chen2010} instead approximate the manifold through patching together hyperplanes.  Such mixtures of linear subspace-based methods may require a large number of subspaces to obtain an accurate approximation even when $d$ is small.

It is clear that probabilistic models for learning the manifold face daunting statistical and computational hurdles.  In this article, we take a very different approach in attempting to define a simple and computationally tractable model, which bypasses the need to estimate $\mathcal{M}$ but can exploit the lower-dimensional manifold structure when it exists.  In particular, our goal is to define an approach that obtains a minimax-optimal adaptive rate in estimating $f$, with the rate adaptive to the manifold and smoothness of the regression function.  Surprisingly, we show that this can be achieved with a simple Gaussian process prior.

Section 2 provides background and our main results. Section 3 discusses two approaches to construct intrinsic dimension adaptive estimators.  Section 4 contains a toy example and a simulation study of finite sample performance relative to competitors. Section 5 provides auxiliary results that are crucial for proving the main results. Technical proofs are deferred to Section 6. A review of necessary geometric properties and selected proofs are included in Section 7 and 8 in the appendix.

\section{Gaussian Processes on Manifolds}


\subsection{Background}
Gaussian processes (GP) are widely used as prior distributions for unknown functions. For example, in the nonparametric regression \eqref{eq:base}, a GP can be specified as a prior for the unknown function $f$. In classification, the conditional distribution of the binary response $Y_i$ is related to the predictor $X_i$ through a known link function $h$ and a regression function $f$ as $Y_i|X_i\sim Ber\big[h\{f(X_i)\}\big]$, where $f$ is again given a GP prior. The following developments will mainly focus on the regression case.
The GP with squared exponential covariance is a commonly used prior in the literature.
The law of the centered squared exponential GP $\{W_x:x\in \mathcal{X}\}$ is entirely determined by its covariance function,
\begin{align}\label{eq:cov}
K^a(x,y)=EW_xW_y=\exp(-a^2||x-y||^2),
\end{align}
where the predictor domain $\mathcal{X}$ is a subset of $\bbR^D$, $||\cdot||$ is the usual Euclidean norm and $a$ is a length scale parameter. Although we focus on the squared exponential case, our results can be extended to a broader class of covariance functions with exponentially decaying spectral density, including standard choices such as Mat\'ern, with some elaboration.
We use $GP(m,K)$ to denote a GP with mean $m:\mathcal{X}\rightarrow \bbR$ and covariance $K: \mathcal{X}\times\mathcal{X}\rightarrow \bbR$.

Given $n$ independent observations, the minimax rate of estimating a $D$-variate function that is only known to be H\"{o}lder $s$-smooth is $n^{-s/(2s+D)}$ \citep{Stone1982}. \cite{Van2009} proved that, for H\"{o}lder $s$-smooth functions, a prior specified as
\begin{align}\label{eq:prior}
W^A|A\sim GP(0,K^A),\quad A^D\sim Ga(a_0,b_0),
\end{align}
for $Ga(a_0,b_0)$ the Gamma distribution with pdf $p(t)\propto t^{a_0-1}e^{-b_0t}$ leads to the minimax rate $n^{-s/(2s+D)}$ up to a logarithmic factor $(\log n)^{\beta}$ with $\beta\sim D$ adaptively over all $s>0$ without knowing $s$ in advance. The superscript in $W^A$ indicates the dependence on $A$, which can be viewed as a scaling or inverse bandwidth parameter.

\begin{figure}[tb]
  \center
  \includegraphics[width=4.5in]{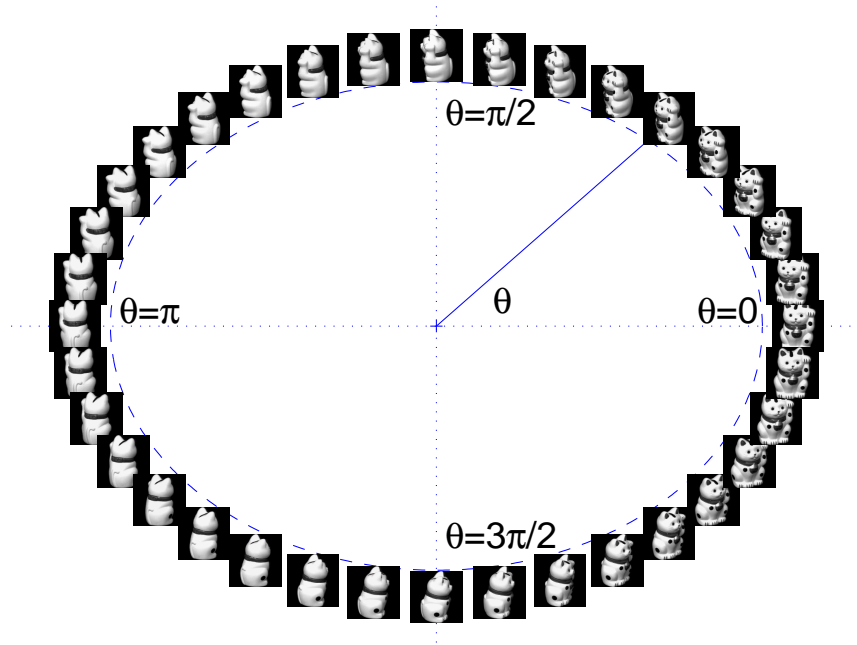}\\
  \caption{In this data, $72$ size $128\times128$ images were taken for a ``lucky cat'' from different angles: one at every $5$ degrees of rotation. $36$ images are displayed in this figure.}\label{fig:data}
\end{figure}

In many real problems, the predictor $X$ can be represented as a vector in high dimensional Euclidean space $\bbR^D$, where $D$ is called the ambient dimensionality.
When $D$ is large, assumptions are required to conquer the notorious curse of dimensionality. One common assumption requires that $f$ only depends on a small number $d\ll n$ of components of the vector $X$ that are identified as important. In the GP prior framework, \cite{Savitsky2011} proposed to use ``spike and slab'' type point mass mixture priors for different scaling parameters for each component of $X$ to do Bayesian variable selection.  Assuming the function is flat in all but $d$ directions, \cite{Anirban2011} showed that a dimension-specific scaling prior for inverse bandwidth parameters can lead to a near minimax rate for anisotropic smooth functions. We instead assume that the predictor lies on a manifold $\mathcal{M}$ of intrinsic dimension $d$ with $d\ll D$. An example is shown in Fig.\ref{fig:data}. These data (\cite{Nene1996}) consist of $72$ images of a ``lucky cat'' taken from different angles $5^{\circ},10^{\circ},\ldots$. The predictor $X\in\bbR^{128^2}$ is obtained by vectorizing the $128\times128$ image. The response $Y$ is a continuous function $f$ of the rotation angle $\theta\in[0,2\pi]$ satisfying $f(0)=f(2\pi)$, such as $sin$ or $cos$ functions. Intuitively, the predictor $X$ concentrates on a circle in $D=128^2$-dim ambient space and thus the intrinsic dimension $d$ of $X$ is equal to one, the dimension of the rotation angle $\theta$.

\subsection{Bayesian regression on manifold}
When $X\in\mathcal{M}$ with the manifold $\mathcal{M}$ $d$-dimensional, a natural question is whether we can achieve the intrinsic rate $n^{-s/(2s+d)}$ for $f$ H\"older $s$-smooth without estimating $\mathcal{M}$.

\citep{Kpotufe2009} and \cite{kpotufe2012} used random projection trees to partition the ambient space and constructed a piecewise constant estimator based on the partition. The authors showed that their estimator has a convergence rate at least $n^{-1/(2+k)}$ for Lipschitz continuous functions that is adaptive to the intrinsic dimension $d$, where $k$ is guaranteed to be of order $O(d\log d)$. A more general framework is considered in \cite{Binev2005} and \cite{Binev2007}, which covers the case where covariates lie on a low dimensional manifold in $\bbR^D$. They studied partition-based estimators and proved an $n^{-r/(2r+1)}$ rate, where $r$ depends on how well the truth $f$ can be approximated by their class. However, it is not clear whether their class of piecewise polynomial functions in $\bbR^D$ can adapt to the manifold structure.

\cite{Ye2008} showed that a least squares regularized algorithm with an appropriate $d$ dependent regularization parameter can ensure a convergence rate at least $n^{-s/(8s+4d)}(\log n)^{2s/(8s+4d)}$ for functions with H\"{o}lder smoothness $s\leq1$. \cite{Bickel2007} proved that local polynomial regression with bandwidth dependent on $d$ can attain the minimax rate $n^{-s/(2s+d)}$ for functions with H\"{o}lder smoothness $s\leq2$.
However, similar adaptive properties have not been established for a Bayesian procedure. In this paper, we will prove that a GP prior on the regression function with a proper prior for the scaling parameter can lead to the minimax rate for functions with H\"{o}lder smoothness $s\leq\{2,\gamma-1\}$, where $\gamma$ is the smoothness of the manifold $\mathcal{M}$. Moreover, we describe two approaches to construct an intrinsic dimension adaptive estimator based on this GP prior in Section \ref{se:in}.
In the remainder of this section, we first propose the model, and then provide a heuristic argument explaining the possibility of manifold adaptivity.

Analogous to \eqref{eq:prior}, we propose the prior for the regression function $f$ as
\begin{align}\label{eq:prior1}
W^A|A\sim GP(0,K^A),\quad A^d\sim Ga(a_0,b_0),
\end{align}
where $d$ is the intrinsic dimension of the manifold $\mathcal{M}$ and $K^a$ is defined as in \eqref{eq:cov} with $||\cdot||$ the Euclidean norm of the ambient space $\bbR^D$. Adaptation to unknown intrinsic dimensionality is considered in Section \ref{se:in}.  Although the GP in \eqref{eq:prior1} is specified through embedding in the $\bbR^D$ ambient space, we essentially obtain a GP on $\mathcal{M}$ if we view the covariance function $K^a$ as a bivariate function defined on $\mathcal{M}\times\mathcal{M}$. Moreover, this prior has two major differences with usual GPs or GP with Bayesian variable selection:
\begin{enumerate}\itemsep=-1pt
  \item Unlike GP with Bayesian variable selection, all predictors are used in the calculation of the covariance function $K^a$;
  \item The dimension $D$ in the prior for inverse bandwidth $A$ is replaced with the intrinsic dimension $d$.
\end{enumerate}

Intuitively, one would expect that geodesic distance should be used in the squared exponential covariance function \eqref{eq:cov}. However, there are two main advantages of using Euclidean distance instead of geodesic distance. First, when geodesic distance is used, the covariance function may fail to be positive definite. In contrast, with Euclidean distance in \eqref{eq:cov}, $K^a$ is ensured to be positive definite. Second, for a given manifold $\mathcal{M}$, the geodesic distance can be specified in many ways through different Riemannian metrics on $\mathcal{M}$. However different geodesic distances are equivalent to each other and to the Euclidean distance on $\bbR^D$. Therefore, by using the Euclidean distance, we bypass the need to estimate geodesic distance, but still reflect the geometric structure of the observed predictors in terms of pairwise distances. In addition, although we use the full data in the calculation of the covariance function, computation is still fast for moderate sample sizes $n$ regardless of the size of $D$ since only pairwise Euclidean distances among $D$-dimensional predictors are involved whose computational complexity scales linearly in $D$.

In this work, we primarily focus on compact manifolds without boundary. The study of manifolds with boundaries is beyond the scope of our current work. A difference occurs because any boundary of a manifold has a smaller dimensionality than the intrinsic dimension of the manifold. As a consequence, in order to achieve optimal rate on boundaries, we may need to consider non-stationary Gaussian process priors, where the length scale parameter $A$ varies on the manifold. However, if we stick to the prior \eqref{eq:prior1}, then we conjecture that the rate is still optimal in the interior, but suboptimal on the boundaries of the manifold.

We provide some heuristic explanations on why the rate can adapt to the predictor manifold. Although the ambient space is $\bbR^D$, the support $\mathcal{M}$ of the predictor $X$ is a $d$ dimension submanifold of $\bbR^{D}$. As a result, the GP prior specified in section 2.1 has all probability mass on the functions supported on this support, leading the posterior contraction rate to entirely depend on the evaluations of $f$ on $\mathcal{M}$. More specifically, the posterior contraction rate is at least $\epsilon_n$ if
\begin{align*}
    \Pi\big(\|f-f_0\|_n>\epsilon_n|S_n\big)\rightarrow0, \ \text{in probability as }  n\rightarrow\infty,
\end{align*}
where $S_n=\{(X_1,Y_1),\ldots,(X_n,Y_n)\}$ denotes the dataset, $\Pi(A|S_n)$ is the posterior of $A$ and $\|f-f_0\|_n$ is the empirical norm defined through $\|f-f_0\|_n^2=(1/n)\sum_{i=1}^n$
$\big(f(x_i)-f_0(x_i)\big)^2$. Hence, $\|f-f_0\|_n$ measures the discrepancy between $f$ and the truth $f_0$, and only depends on the evaluation of $f$ on $\mathcal{M}$. Therefore, in a prediction perspective, we only need to fit and infer $f$ on $\mathcal{M}$. Intuitively, we can consider a special case when the points on manifold $\mathcal{M}$ have a global smooth representation $x=\phi(t)$, where $t\in\bbR^d$ is the global latent coordinate of $x$. Then the regression function
\begin{align}\label{eq:para}
 f(x)=f\big[\phi(t)\big]\triangleq h(t),\ \ t\in\bbR^d,
\end{align}
is essentially a $d$-variate $s$-smooth function if $\phi$ is sufficiently smooth. Then estimation of $f$ on $\bbR^D$ boils down to estimation of $h$ on $\bbR^d$ and the intrinsic rate would be attainable.

\subsection{Convergence rate under fixed design}
The following theorem is our main result which provides posterior convergence rate under fixed design.

\begin{theorem}\label{thm:main1}
Assume that $\mathcal{M}$ is a $d$-dimensional compact $C^{\gamma}$ submanifold of $R^{D}$. For any $f_0\in C^{s}(\mathcal{M})$ with $s\leq \min\{2,\gamma-1\}$, if we specify the prior as \eqref{eq:prior}, then (\ref{eq:0c}) below will be satisfied for
$\epsilon_n$ a multiple of $n^{-s/(2s+d)}(\log n)^{\kappa_1}$ and $\bar{\epsilon}_n$ a multiple of $\epsilon_n(\log n)^{\kappa_2}$ with $\kappa_1=(1+d)/(2+d/s)$ and $\kappa_2=(1+d)/2$. This implies that the posterior contraction rate with respect to $\|\cdot\|_n$ will be at least a multiple of $n^{-s/(2s+d)}(\log n)^{d+1}$.
\end{theorem}

The ambient space dimension $D$ implicitly influences the rate via a multiplicative constant. This theorem suggests that the Bayesian model \eqref{eq:prior1} can adapt to both the low dimensional manifold structure of $X$ and the smoothness $s\leq 2$ of the regression function. The reason the near optimal rate can only be allowed for functions with smoothness $s\leq2$ is the order of error in approximating the intrinsic distance $d_{\mathcal{M}}$ by the Euclidean distance $d$ (Proposition 7.5).

Generally, the intrinsic dimension $d$ is unknown and needs to be estimated.
In the case when the intrinsic dimensionality $d$ is
misspecified as $d'$, the following corollary still ensures the rate to be much better than $n^{-O(1/D)}$ when $d'$ is not too small, although the rate becomes suboptimal.
\begin{corollary}\label{thm:main1b}
Assume the same conditions as in Theorem \ref{thm:main1}, but with the prior specified as \eqref{eq:prior} with $d'\neq d$ and $d'>d^2/(2s+d)$.
\begin{enumerate}\itemsep=-1pt
  \item If $d'>d$, then the posterior contraction rate with respect to $\|\cdot\|_n$ will be at least a multiple of $n^{-s/(2s+d')}(\log n)^{\kappa}$, where
      $\kappa=(1+d)/(2+d'/s)$;
  \item If $\frac{d^2}{2s+d}<d'<d$, then the posterior contraction rate with respect to $\|\cdot\|_n$ will be at least a multiple of $n^{-\frac{(2s+d)d'-d^2}{2(2s+d)d'}}(\log n)^{\kappa}$, where
      $\kappa=(d+d^2)/(2d'+dd'/s)+(1+d)/2$.
\end{enumerate}
\end{corollary}

\subsection{Convergence rate under random design}\label{se:rd}

Theorem \ref{thm:main1} obtains the posterior contraction rate for fixed design.
In general, convergence rate in a random designed case is more challenging.
\cite{van2011} obtains posterior convergence rates for regression on Euclidean space $\bbR^d$ using GP priors. However, they require $s\geq d/2$ for estimating an $s$-smooth function. This assumption restricts the applicability of Theorem \ref{thm:main1} as it assumes $s\leq 2$. \cite{van2011} also makes a crucial assumption that the prior puts all its mass over $s$-smooth function spaces, leading to suboptimal rates of estimating non-analytic functions with the squared exponential covariance kernel.

Instead of directly proving that Theorem \ref{thm:main1} works for the random design case, we take a different approach by post-processing the posterior.
We show that the resulting posterior can achieve the same rate $\bar{\epsilon}_n$ with respect to $\|\cdot\|_2$, and the corresponding Bayes estimator $\hat{f}$ satisfies $\|\hat{f}-f_0\|_{2}\lesssim \bar{\epsilon}_n$ with high probability. Here $\|f\|_2 \triangleq\int_{\mathcal{M}}f^2(x)G(dx)$, with $G$ the marginal distribution for predictor $X$. Usual empirical process theory (Lemma 6.2, or \cite{geer2000}) requires the function space to be uniformly bounded in order for the empirical norm $\|\cdot\|_n$ to be comparable to the $L_2$ norm $\|\cdot\|_2$ uniformly over the space. This motivates us to truncate the functions sampled from the posterior distribution. The idea of post-processing a posterior has also been considered in \cite{Lin2014} for Bayesian monotone regression, which projects posterior samples to the monotone function space.

Let $A$ be any upper bound for $\|f_0\|_{\infty}$. For any function $f$, denote its truncation by $A$ as $f_A=(f\vee(-A))\wedge A$. Then our post-processed posterior is the posterior of $f_A$ and the corresponding estimator $\hat{f}$ is given by $\hat{f}(x)=\int f_A(x)d\Pi(f|S_n)$, which is the posterior expectation of $f_A$. In practice, $\hat{f}$ can be easily obtained by taking average of $\{f^{(j)}_A:j=1,\ldots,N\}$ where $\{f^{(j)}:j=1,\ldots,N\}$ are sampled from the posterior distribution of $f$.
The reason for truncating $f$ in the posterior is two-fold. For practical purposes, truncation will never deteriorate an estimator, i.e. $|f_A(x)-f_0(x)|\leq |f(x)-f_0(x)|$ for all $x$ as long as $A\geq\|f_0\|_{\infty}$. For theoretical purposes, we require the estimator to be bounded in order to compare $\|\cdot\|_n$ and $\|\cdot\|_2$ by applying results in empirical process theory.

The following theorem shows that the truncated GP posterior contracts to $f_0$ at a near minimax-optimal rate with respect to both $\|\cdot\|_n$ and $\|\cdot\|_2$. Moreover, the corresponding estiamtor $\hat{f}$ is near minimax-optimal under both fixed design and random design. A proof is provided in Section \ref{se:tp}.

\begin{theorem}\label{le:BEr}
Assume the same conditions as in Theorem \eqref{thm:main1}, then
\begin{align*}
    \Pi\big(\max\{\|f_A-f_0\|_n,\|f_A-f_0\|_2>\epsilon_n\big|S_n\big)\rightarrow0, \ \text{in probability as }  n\rightarrow\infty.
\end{align*}
Moreover with probability tending to one, the following holds:
\begin{align*}
    \max\big\{\|\hat{f}-f_0\|_n,\|\hat{f}-f_0\|_2\big\}\leq C\bar{\epsilon}_n\leq n^{-s/(2s+d)}(\log n)^{d+1},
\end{align*}
with $\bar{\epsilon}_n$ given in Theorem \eqref{thm:main1} and $C$ a positive constant.
\end{theorem}

\subsection{Dimensionality Reduction}
\cite{Tenenbaum2000} and \cite{Roweis2000} initiated the area of manifold learning, which aims to design non-linear dimensionality reduction algorithms to map high dimensional data into a low dimensional feature space under the assumption that data fall on an embedded non-linear manifold within the high dimensional ambient space. A combination of manifold learning and usual nonparametric regression leads to a two-stage approach, in which a dimensionality reduction map from the original ambient space $\bbR^D$ to a feature space $\bbR^{\tilde{d}}$ is estimated in the first stage and a nonparametric regression analysis with low dimensional features as predictors is conducted in the second stage. As a byproduct of Theorem \ref{thm:main1}, we provide a theoretical justification for this two stage approach under some mild conditions.

Let $\Psi:\bbR^D\rightarrow \bbR^{\tilde{d}}$ be a dimensionality reduction map. For identifiability, we require the restriction $\Psi_{\mathcal{M}}$ of $\Psi$ on the manifold $\mathcal{M}$ to be a diffeomorphism, i.e. $\Psi_{\mathcal{M}}$ is injective and both $\Psi_{\mathcal{M}}$ and its inverse are smooth, which requires $\tilde{d}\geq d$. Diffeomorphism is the least and only requirement such that both the intrinsic dimension $d$ of predictor $X$ and smoothness $s$ of regression function $f$ are invariant. If we view $\Psi(\bbR^D)$ as the new ambient space, then the new regression function $\tilde{f}$ is induced by $f$ via
\begin{align*}
    \tilde{f}(\tilde{x})=f\big[\Psi_{\mathcal{M}}^{-1}(\tilde{x})\big],\text{ for all }\tilde{x}\in\Psi_{\mathcal{M}}(\mathcal{M}).
\end{align*}
Accordingly, the empirical norm of $\tilde{f}$ under fixed design becomes $\|\tilde{f}\|_n^2=\sum_{i=1}^n|\tilde{f}(\Psi(X_i))|^2$.
By the identifiability condition of $\Psi$, $\tilde{f}$ is a well defined function on the manifold $\mathcal{M}$ represented in ambient space $\bbR^{\tilde{d}}$ and has the same smoothness as $f$. Therefore, by specifying a GP prior \eqref{eq:prior} directly on $\bbR^{\tilde{d}}$, we would be able to achieve a posterior contraction rate at least $n^{-s/(2s+d)}(\log n)^{d+1}$, as indicated by the following corollary.
\begin{corollary}\label{thm:main2}
Assume that $\mathcal{M}$ is a $d$-dimensional compact $C^{\gamma}$ submanifold of $R^{D}$. Suppose that $\Psi:\bbR^D\rightarrow\bbR^{\tilde{d}}$ is an ambient space mapping (dimension reduction) such that $\Psi$ restricted on $\mathcal{M}$ is a $C^{\gamma'}$-diffeomorphism. Then by specifying the prior \eqref{eq:prior} with $\{\Psi(X_i)\}_{i=1}^n$ as observed predictors and Euclidean norm of $\bbR^{\tilde{d}}$ as $||\cdot||$ in \eqref{eq:cov}, for any $f_0\in C^{s}(\mathcal{M})$ with $s\leq \min\{2,\gamma-1,\gamma'-1\}$, (\ref{eq:0c}) will be satisfied for
$\epsilon_n=n^{-s/(2s+d)}(\log n)^{\kappa_1}$ and $\bar{\epsilon}_n=\epsilon_n(\log n)^{\kappa_2}$ with $\kappa_1=(1+d)/(2+d/s)$ and $\kappa_2=(1+d)/2$. This implies that the posterior contraction rate with respect to $\|\cdot\|_n$ will be at least $\epsilon_n=n^{-s/(2s+d)}(\log n)^{d+1}$.
\end{corollary}

\section{Adaptation to intrinsic dimension}\label{se:in}

To make our approach adaptive to the intrinsic dimension, we can follow an empirical Bayes approach and plug in an estimator of the dimension.  Such an estimator can be chosen focusing either on inference on $d$ or on prediction.  In the latter approach, the estimator for $d$ may not be consistent but one can achieve a near minimax-optimal rate.  Focusing on our truncated estimator in the random design case, we describe two approaches in the next subsections.

\subsection{Intrinsic dimension estimation}\label{se:pest}
Since $d$ serves as a hyper-parameter in prior \eqref{eq:prior1}, in principle one can specify a prior for $d$ over a finite grid $d_1\leq\ldots\leq d_{p}$ and conditioning on $d=d_j$, use \eqref{eq:prior1} as a prior for $f$. Since $W^A$ is conditionally independent of $d$ given $A$, one can marginalize out $d$ and obtain an equivalent prior for $A$ as a mixture distribution. In the proof of Theorem \eqref{thm:main1}, the only property of the prior of $A$ that is used is its tail behavior as $P(A>a)\sim \exp(-Ca^d)$. However, with an extra level of prior for $d$, the tail marginal prior probability $P(A>a)$ is dominated by $\exp(-Ca^{d_1})$, which has similar decay rate as the prior $A^{d_1}\sim Ga(a_0,b_0)$. As a consequence, specifying a prior for $d$ leads to sub-optimal rate as indicated by Corollary \ref{thm:main1b}.

Intuitively, information on the intrinsic dimension $d$ is contained in the marginal distribution of $X$, which cannot be fully revealed by estimating the conditional distribution $P(Y|X)$. This motivates our first approach of estimating $d$ directly based on the covariates $\{X_i\}$.

Many estimation methods have been proposed for determining the intrinsic dimension of a dataset lying on a manifold \citep{Carter2010,Farahmand07,Camastra2002,Levina2004,Little2009}. For example, \cite{Levina2004} considers a likelihood based approach and \cite{Little2009} relies on singular value decomposition of the local sample covariance matrix. \cite{Farahmand07} proposes a nearest-neighbor method and analyzes its finite-sample properties.
Their estimator $\hat{d}$ takes the form as
\begin{align*}
    \hat{d}=\frac{\log 2}{\log\hat{r}^{(k)}(X_1)-\log \hat{r}^{(\lceil k/2\rceil)}(X_1)},
\end{align*}
where $\hat{r}^{(k)}(X_1)$ is the $k$th nearest neighbor of $X_1$ in $\{X_i\}$. They proved that under some mild conditions on the distribution of $X$, if $n\geq k2^d$, then with probability at least $1-\delta$,
\begin{align*}
    |\hat{d}-d|\leq C\bigg\{\bigg(\frac{k}{n}\bigg)^{\frac{1}{d}}+\sqrt{\frac{\log(4/\delta)}{k}}\bigg\},
\end{align*}
where $C$ is some constant independent of $k$ and $n$. As a consequence, if we choose $k=n^{-1/2}$ and  let $\hat{d}_R$ be the closest integer to $\hat{d}$, then $P_0(\hat{d}_R\neq d)\to0$ as $n\to\infty$.

We will use \cite{Farahmand07} to obtain an estimator of $d$ as $\hat{d}_R$ and then plug in $\hat{d}_R$ into our prior \eqref{eq:prior1} to obtain an empirical Bayes estimator $\hat{f}_{EB}$ as in Section \ref{se:rd}. The following corollary summarizes its asymptotic performance.

\begin{corollary}
    Assume Assumption 1 in \cite{Farahmand07} and the same conditions as in Theorem \eqref{thm:main1}, then with probability tending to one,
\begin{align*}
    \max\big\{\|\hat{f}_{EB}-f_0\|_n,\|\hat{f}_{EB}-f_0\|_2\big\}\lesssim n^{-s/(2s+d)}(\log n)^{d+1}.
\end{align*}
\end{corollary}

\subsection{Cross validation}
In this subsection, we select a best dimension and its associated estimator as constructed in Section \ref{se:rd} based on prediction accuracy on a testing set. This selection rule cannot consistently estimate $d$ but still yields an optimal convergence rate (Theorem \ref{thm:crv}). In principle, the selection procedure described in this subsection can be applied to any hyperparameter selection problem.

We focus on the random design case. Let $d_{max}$ be a pre-specified upper bound for $d$. For example, we can choose $d_{max}=20$. Let $\Pi_k$ be the prior \eqref{eq:prior1} with $d=k$ for $k=1,\ldots,d_{max}$. The selection procedure proceeds as follows:
\begin{enumerate}
  \item Randomly split the data set with sample size $n+m$ into a training set $S_n=\{(X_i,Y_i):i=1,\ldots,n\}$ and a testing set $\tilde{S}_m=\{(\tilde{X}_i,\tilde{Y}_i):i=1,\ldots,m\}$.
  \item For $k=1,\ldots,d_{max}$, obtain a truncated Bayes estimator $\hat{f}^{(k)}$ under $\Pi_k$ defined as $\int f_A(x)d\Pi_k(f|S_n)$ for all $x\in \mathcal{M}$ as in Section \ref{se:rd}. Compute its mean squared prediction error (MSPE) $E_m^{(k)}=m^{-1}\sum_{i=1}^m\big(\hat{f}^{(k)}(\tilde{X}_i)-\tilde{Y}_i\big)^2$
      on the testing set.
  \item Let $\hat{d}_{CV}=\text{arg}\min_k E_m^{(k)}$. The final estimator is defined by $\hat{f}_{CV}=\hat{f}^{(\hat{d}_{CV})}$.
\end{enumerate}

The intuition is simple: an estimator with minimal MSPE, which approximately minimizes $\|\hat{f}^{(k)}-f_0\|^2_2$ over $k$, should be at least better than $\hat{f}^{(d)}$, the estimator under the true dimensionality $d$.
In practice, one can repeat step 1 and 2 for a number of times and use an averaged MSPE instead of $E_m^{(k)}$ to improve stability. However, the following theorem suggests that one splitting suffices for the adaptivity on the dimensionality.

\begin{theorem}\label{thm:crv}
Suppose $d\leq d_{max}$ and $A\geq \|f_0\|_{\infty}$ in the cross validation procedure. If $m\min_k\|\hat{f}^{(k)}-f_0\|_2^2\to \infty$, then under the conditions in Theorem \ref{thm:main1},
\begin{align*}
   \|\hat{f}_{CV}-f_0\|_2\lesssim n^{-s/(2s+d)}(\log n)^{d+1}.
\end{align*}
\end{theorem}

The only condition on $m$ is $m\min_k\|\hat{f}^{(k)}-f_0\|_2^2\to \infty$, which guarantees $\|\hat{f}^{(k)}-f_0\|_2$ and $\|\hat{f}^{(k)}-f_0\|_m$ to be close for all $k$, where $\|\cdot\|_m$ denotes the empirical $L_2$ norm on the testing set. As a consequence, this condition allows us to use $E_m^{(k)}=\|\hat{f}^{(k)}-f_0\|_m^2+R_m^{(k)}$ to approximate $\|\hat{f}^{(k)}-f_0\|_2^2$, where the variation of the remainder term $R_m^{(k)}$ across $k$ is asymptotically negligible. Since we expect $\min_k\|\hat{f}^{(k)}-f_0\|_2^2\approx \|\hat{f}^{(d)}-f_0\|_2^2\approx\|\hat{f}^{(d)}-f_0\|_n^2\asymp$ $n^{-2s/(2s+d)}(\log n)^{d+1}$, $m$ might be chosen as $O(n^{\gamma})$ for any $\gamma>2s/(2s+d)$. In practice, one can simply choose $m=O(n)$.

\section{Numerical Examples}

\subsection{Regression on the Swiss Roll}
We start with a toy example where $X$ lies on a two-dimensional Swiss roll in the $100$-dimensional Euclidean space (Fig.~\ref{fig:sr} plots a typical Swiss roll in the three-dimensional Euclidean space). $X$ is generated as follows. We first sample $T=(T_1,T_2,T_3)^T$ from a two-dimensional Swiss roll in three-dimensional ambient space as
\begin{align*}
    T_1=U\cos(U), \ T_2=V, \ T_3=U\sin(U),
\end{align*}
with $U\sim Unif\big(\frac{3\pi}{2},\frac{9\pi}{2}\big)$ and $V\sim Unif(0,20)$. Then we transform $T$ into a $100$-dimensional vector via $X=\Omega T$, where $\Omega$ is a randomly generated $100$-by-$3$ matrix whose components follow iid $N(0,1)$. $\Omega$ will be fixed in each synthetic dataset. The response $Y$ depends on $X$ through
\begin{align*}
    Y=4\bigg(\frac{1}{3\pi}U-\frac{1+3\pi}{2}\bigg)^2+\frac{\pi}{20} V+ N(0,0.1^2).
\end{align*}
To assess the fitting performance, we use the empirical error $\|\hat{f}-f_0\|_n$ of our estimator $\hat{f}$ defined in Section \ref{se:rd} on the design points as a criterion. Here $f_0(x)=4(u-0.5)^2+\pi v$.
In the GP approach, we apply the empirical Bayes approach described in Section \ref{se:pest} and run $10,000$ iterations with the first $5,000$ as burn-in in each replicate.
We report an average empirical error (AEE) over $100$ replicates in Table~\ref{tab:2}. In this example, the GP estimator has a relatively fast convergence rate even though the dimensionality of the ambient space is large, which justifies our theory.

\begin{figure}[tb]
  \center
  \includegraphics[width=4in]{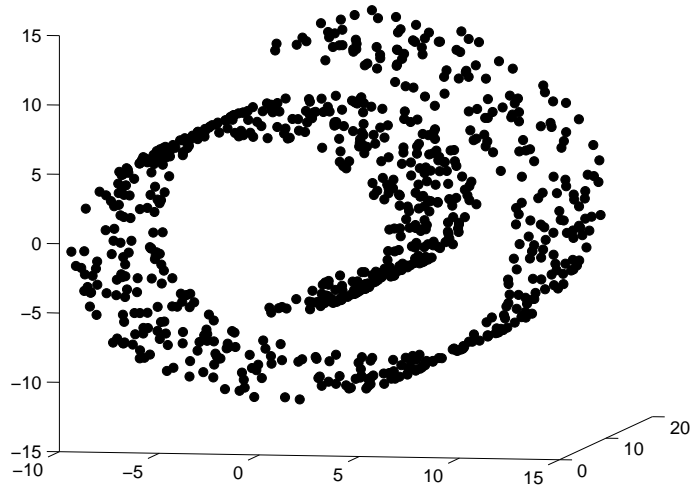}\\
  \caption{A typical Swiss roll in three-dimensional Euclidean space.}\label{fig:sr}
\end{figure}

\begin{table}
\centering
\caption{Simulation results for the Swiss roll example over $100$ replicates. The numbers in the parentheses indicate standard deviations.}
\begin{tabular}{|c|ccccc|}
    \hline
   &$n=50$ & $n=100$ & $n=200$ & $n=400$ & $n=800$ \\
    \hline
   AEE & .164(.090) & .143(.026) & .121(.012) & .106(.005) & .095(.003)\\
    \hline
\end{tabular}
\label{tab:2}
\end{table}

\subsection{Application to the lucky cat data}
The lucky cat data (Fig.~\ref{fig:data}) has intrinsic dimensionality one, which is the dimension of the rotation angle $\theta$. Since we know the true value of $\theta$, we create the truth $f_0(\theta)=\cos\theta$ as a continuous function on the unit circle. The responses are simulated from $Y_i=f_0(\theta_i)+\epsilon_i$ by adding independent Gaussian noises $\epsilon_i\sim N(0,0.1^2)$ to the true values. In this model, the total sample size $N=72$ and the predictors $X_i\in\bbR^{p}$ with $D=16,384$. To assess the impact of the sample size $n$ on the fitting performance, we randomly divide $n=18$, $36$ and $64$ samples into training set and treat the rest as testing set. Training set is used to fit a model and testing set to quantify the estimation accuracy. For each training size $n$, we repeat this procedure for $m=100$ times and calculate the square root of mean squared prediction error (MSPE) on the testing set,
\begin{align*}
    \sum_{l=1}^m \frac{1}{N-n}\sum_{i\in T_l}||\hat{Y}_i-f_0(\theta_i)||^2,
\end{align*}
where $T_l$ is the $l$th testing set and $\hat{Y}_i$ is an estimation of $E[Y|X_i]=f_0(\theta_i)$. We apply two GP based algorithms on this data set: 1. vanilla GP specified by \eqref{eq:prior1}; 2. Two stage GP (2GP) where the $D$-dimensional predictors were projected into $\bbR^2$ by using Laplacian eigenmap \citep{Belkin2003} in the first stage and then a GP with projected features as predictors was fitted in the second stage. To assess the prediction performance, we also compare our GP prior based models \eqref{eq:prior1} with lasso \citep{Tibshirani1996} and elastic net (EN) \citep{Zou2005} under the same settings. We choose these two competing models because they are among the most widely used methods in high dimensional regression settings and perform especially good when the true model is sparse. In the GP models, we set $d=1$ since the sample size for this dataset is too small for most dimension estimation algorithms to reliably estimate $d$. In addition, for each simulation, we run $10,000$ iterations with the first $5,000$ as burn-in.
\begin{table}
\centering
\caption{Square root of MSPE for the lucky cat data by using two different approaches over $100$ random splitting are displayed. The numbers in the parentheses indicate standard deviations.}
\begin{tabular}{|c|ccc|}
    \hline
    & $n=18$ & $n=36$ & $n=54$ \\
    \hline
  EN & .416(.152) & .198(.042) & .149(.031) \\
  LASSO & .431(.128) & .232(.061) & .163(.038) \\
  GP & .332(.068) & .128(.036) & \textbf{.077(.014)} \\
  2GP & \textbf{.181(.051)} &  \textbf{.124(.038)} & .092(.021) \\
   \hline
\end{tabular}
\label{tab:1}
\end{table}

The results are shown in Table.~\ref{tab:1}. As we can see, under each training size $n$, GP performs the best. Moreover, as $n$ increases, the prediction error of GP decays much faster than EN and Lasso: when $n=18$, the square root of MSPEs by using EN and lasso are about 125\% of that by using GP; however as $n$ increases to $54$, this ratio becomes about 200\%. Moreover, the standard deviation of square root of MSPEs by using GP are also significantly lower than those by using lasso and EN. It is not surprising that 2GP has better performance than GP when $n$ is small since the dimensionality reduction map $\Psi$ is constructed using the whole dataset (the Laplacian eigenmap code we use cannot do interpolations). Therefore when the training size $n$ become closer to the total data size $72$, GP becomes better. In addition, GP is computationally faster than 2GP due to the manifold learning algorithm in the first stage of 2GP.

\section{Auxiliary results}

In the GP prior \eqref{eq:prior1}, the covariance function $K^a:\mathcal{M}\times \mathcal{M}\rightarrow R$ is essentially defined on the submanifold $\mathcal{M}$. Therefore, \eqref{eq:prior1} actually defines a GP on $\mathcal{M}$ and we can study its posterior contraction rate as a prior for functions on the manifold. In this section, we combine geometry properties and Bayesian nonparametric asymptotic theory to prove the theorems in section 2.

\subsection{Reproducing Kernel Hilbert Space on the Manifold}
Being viewed as a covariance function defined on $[0,1]^D\times [0,1]^D$, $K^a(\cdot,\cdot)$ corresponds to a reproducing kernel Hilbert space (RKHS) $\mathbb{H}^{a}$, which is defined as the completion of $\mathcal{H}$, the linear space of all functions on $[0,1]^D$ with the following form
\[
x\mapsto\sum_{i=1}^{m}a_iK^a(x_i,x), x\in[0,1]^D,
\]
indexed by $a_1,\ldots,a_m\in \bbR$ and $x_1,\ldots,x_m\in[0,1]^D,\ m\in\mathds{N}$, relative to the norm induced by the inner product defined through $\langle K^a(x,\cdot), K^a(y,\cdot)\rangle_{\mathbb{H}^{a}}=K^a(x,y)$. Similarly, $K^a(\cdot,\cdot)$ can also be viewed as a covariance function defined on $\mathcal{M}\times \mathcal{M}$, with the associated RKHS denoted by $\tilde{\mathbb{H}}^{a}$. Here $\tilde{\mathbb{H}}^{a}$ is the completion of $\tilde{\mathcal{H}}$, which is the linear space of all functions on $\mathcal{M}$ with the following form
  \[
x\mapsto\sum_{i=1}^{m}a_iK^a(x_i,x), x\in\mathcal{M},
\]
indexed by $a_1,\ldots,a_m\in \bbR$ and $x_1,\ldots,x_m\in\mathcal{M},\ m\in\mathds{N}$.

Many probabilistic properties of GPs are closely related to the RKHS associated with its covariance function. Readers can refer to \cite{Aronszajn1950} and \cite{Van2008b} for introductions on RKHS theory for GPs on Euclidean spaces. In order to generalize RKHS properties in Euclidean spaces to submanifolds, we need a link to transfer the theory. The next lemma achieves this by characterizing the relationship between $\mathbb{H}^{a}$ and $\tilde{\mathbb{H}}^{a}$.
\begin{lemma}\label{le:1a}
For any $f\in \tilde{\mathbb{H}}^{a}$, there exists $g\in\mathbb{H}^{a}$ such that $g|_\mathcal{M}=f$ and $||g||_{\mathbb{H}^{a}}=||f||_{\tilde{\mathbb{H}}^{a}}$, where
$g|_\mathcal{M}$ is the restriction of $g$ on $\mathcal{M}$. Moreover, for any other $g'\in\mathbb{H}^{a}$ with $g'|_\mathcal{M}=f$, we have $||g'||_{\mathbb{H}^{a}}\geq||f||_{\tilde{\mathbb{H}}^{a}}$, which implies $||f||_{\tilde{\mathbb{H}}^{a}}^=\inf_{g\in\mathbb{H}^{a},g|_\mathcal{M}=f}||g||_{\mathbb{H}^{a}}$.
\end{lemma}

This lemma implies that any element $f$ in the RKHS $\tilde{\mathbb{H}}^{a}$ could be considered as the restriction of some element $g$ in the RKHS $\mathbb{H}^{a}$. Particularly, there exists a unique such element $g$ in $\mathbb{H}^{a}$ such that the norm is preserved, i.e. $||g||_{\mathbb{H}^{a}}=||f||_{\tilde{\mathbb{H}}^{a}}$.

\subsection{Background on Posterior Convergence Rate for GP}
As shown in \cite{Ghosal2000}, in order to characterize the posterior contraction rate in a Bayesian nonparametric problem, such as density estimation, fixed/random design regression or classification, we need to verify some conditions on the prior measure $\Pi$. Specifically, we describe the sufficient conditions for randomly rescaled GP prior as \eqref{eq:prior} given in \cite{Van2009}. Let $\mathcal{X}$ be the predictor space and $f_0$ be the true function $f_0:\mathcal{X}\rightarrow \bbR$, which is the log density $\log p(x)$ in density estimation, regression function $E[Y|X]$ in regression and logistic transformed conditional probability $\text{logit} P(Y=1|X)$ in classification. We will not consider density estimation since to specify the density by log density $f_0$, we need to know the support $\mathcal{M}$ so that $e^{f_0}$ can be normalized to produce a valid density. Let $\epsilon_n$ and $\bar{\epsilon}_n$ be two sequences. If there exist Borel measurable subsets $B_n$ of $C(\mathcal{X})$ and constant $K>0$ such that for $n$ sufficiently large,
\begin{equation}\label{eq:0c}
\begin{aligned}
P(||W^A-f_0||_{\infty}\leq\epsilon_n)&\geq e^{-n\epsilon_n^2},\\
P(W^A\notin B_n)&\leq e^{-4n\epsilon_n^2},\\
\log N(\bar{\epsilon}_n,B_n,||\cdot||_{\infty})&\leq n\bar{\epsilon}_n^2,
\end{aligned}
\end{equation}
where $W^A\sim \Pi$ and $||\cdot||_{\infty}$ is the sup-norm on $C(\mathcal{X})$, then the posterior contraction rate would be at least $\epsilon_n\vee\bar{\epsilon}_n$ under $\|\cdot\|_n$. In our case, $\mathcal{X}$ is the $d$-dimensional submanifold $\mathcal{M}$ in the ambient space $\bbR^D$. We require $\mathcal{M}$ to be compact because the space of continuous functions on a compact metric space is a separable Banach space, which is fundamental to apply the theory from \cite{Van2009}.
To verify the first concentration condition, we need to give upper bounds to the so-called concentration function \citep{Van2009} $\phi_{f_0}^a(\epsilon)$ of the GP $W^a$ around truth $f_0$ for given $a$ and $\epsilon$. $\phi_{f_0}^a(\epsilon)$ is composed of two terms. Both terms depend on the RKHS $\tilde{\mathbb{H}}^{a}$ associated with the covariance function of the GP $W^a$. The first term is the decentering function $\inf\{||h||_{\tilde{\mathbb{H}}^{a}}^2:||h-f_0||_{\infty}<\epsilon\}$, where $||\cdot||_{\tilde{\mathbb{H}}^{a}}$ is the RKHS norm. This quantity measures how well the truth $f_0$ could be approximated by the elements in the RKHS. The second term is the negative log small ball probability $-\log P(||W^a||_{\infty}<\epsilon)$, which depends on the covering entropy $\log N(\epsilon_n,\tilde{\mathbb{H}}^{a}_1,||\cdot||_{\infty})$ of the unit ball in the RKHS $\tilde{\mathbb{H}}^{a}$. As a result of this dependence, by applying Borell's inequality \citep{Van2008b}, the second and third conditions can often be proved as byproducts by using the conclusion on the small ball probability.

As pointed out by \cite{Van2009}, the key to ensure the adaptability of the GP prior on Euclidean spaces is a sub-exponential type tail of its stationary covariance function's spectral density, which is true for squared exponential and Mat\'ern class covariance functions.
More specifically, a squared exponential covariance function $K_1(x,y)=\exp\big\{-||x-y||^2/2\big\}$ on $\bbR^D$ has a spectral representation as
  $$K_1(x,y)=\int_{\bbR^D} e^{-i(\lambda,x-y)}\mu(d\lambda),$$
where $\mu$ is its spectral measure with a sub-Gaussian tail, which is lighter than sub-exponential tail in the sense that: for any $\delta>0$,
\begin{equation}\label{eq:1}
    \int e^{\delta||\lambda||}\mu(d\lambda)<\infty.
\end{equation}
For convenience, we will focus on squared exponential covariance function, since generalizations to other covariance functions with sub-exponential decaying spectral densities are possible with more elaboration.

\subsection{Decentering Function}
To estimate the decentering function, the key step is to construct a function $I_a(f)$ on the manifold $\mathcal{M}$ to approximate a differentiable function $f$, so that the RKHS norm $||I_a(f)||_{\tilde{\mathbb{H}}^{a}}$ can be tightly upper bounded. Unlike in Euclidean spaces where functions in the RKHS $\mathbb{H}^{a}$ can be represented via Fourier transformations \citep{Van2009}, there is no general way to represent and calculate RKHS norms of functions in the RKHS $\tilde{\mathbb{H}}^{a}$ on a manifold. Therefore in the next lemma we provide a direct way to construct the approximation function $I_a(f)$ for any truth $f$ via convolving $f$ with $K^a$ on manifold $\mathcal{M}$:
\begin{align}
I_a(f)(x)&=\bigg(\frac{a}{\sqrt{2\pi}}\bigg)^d\int_\mathcal{M}K^a(x,y)f(y)dV(y)\nonumber\\
&=\bigg(\frac{a}{\sqrt{2\pi}}\bigg)^d\int_\mathcal{M}\exp\bigg\{-\frac{a^2||x-y||^2}{2}\bigg\}f(y)dV(y), \quad x\in\mathcal{M}, \label{eq:6b}
\end{align}
where $V$ is the Riemannian volume form of $\mathcal{M}$. Heuristically, for large $a$, the above integrand only has non-negligible value in a small neighborhood around $x$. Therefore we can conduct a change of variable in the above integral with transformation $\phi^x:B_{\delta}\rightarrow W$ defined by (7.2) in the appendix in a small neighborhood $W$ of $x$:
\begin{align*}
    I_a(f)(x)&=\bigg(\frac{a}{\sqrt{2\pi}}\bigg)^d\int_{\bbR^d}
    \exp\bigg\{-\frac{a^2||\phi^x(u)-\phi^x(0)||^2}{2}\bigg\}\\
    &\qquad\qquad\qquad\qquad\qquad\qquad f\big(\phi^x(u)\big)
    \sqrt{\det(g^{\phi}_{ij}(u))}du,\\
    &\approx\bigg(\frac{a}{\sqrt{2\pi}}\bigg)^d\int_{\bbR^d}
    \exp\bigg\{-\frac{a^2||u||^2}{2}\bigg\}f\big(\phi^x(u)\big)du, \\
    &\approx f\big(\phi^x(0)\big)=f(x),\quad x\in\mathcal{M},
\end{align*}
where the above approximation holds since: 1. $\phi^x(0)=x$; 2. $\phi^x$ preserve local distances (Appendix, Proposition 7.5\,(3)); 3. the Jacobian $\sqrt{\det(g^{\phi}_{ij}(u))}$ is close to one (Appendix, Proposition 7.5\,(2)). From this heuristic argument, we can see that the approximation error $||I_a(w)-f_0||_{\infty}$ is determined by two factors: the convolution error $\big|\big(\frac{a}{\sqrt{2\pi}}\big)^d\int_{\bbR^d}
\exp\big\{-\frac{a^2||u||^2}{2}\big\}f\big(\phi^x(u)\big)du-f(x)\big|$ and the non-flat error caused by the nonzero curvature of $\mathcal{M}$.  Moreover, we can expand each of these errors as a polynomial of $1/a$ and call the expansion term related to $1/a^k$ as $k$th order error.

When $\mathcal{M}$ is Euclidean space $\bbR^d$, the non-flat error is zero, and by Taylor expansion the convolution error has order $s$ if $f_0\in C^s(\bbR^d)$ and $s\leq 2$, where $C^s(\bbR^d)$ is the Holder class of $s$-smooth functions on $\bbR^d$. This is because the Gaussian kernel $\exp\{-||(x-y)||^2/2\}$ has a vanishing moment up to first order: $\int x\exp(-||(x-y)||^2/2)dx=0$. Generally, the convolution error could have order up to $s+1$ if the convolution kernel $K$ has vanishing moments up to order $s$, i.e. $\int x^{t}K(x)dx=0,t=1,\ldots,s$. However, for general manifold $\mathcal{M}$ with non-vanishing curvature tensor, the non-flat error always has order two (see the proof of Lemma \ref{le:7}). This implies that even though carefully chosen kernels for the covariance function can improve the convolution error to have order higher than two, the overall approximation still tends to have second order error due to the deterioration caused by the nonzero curvature of the manifold. The following lemma formalizes the above heuristic argument on the order of the approximation error by \eqref{eq:6b} and further provides an upper bound on the decentering function.

\begin{lemma}\label{le:7}
Assume that $\mathcal{M}$ is a $d$-dimensional compact $C^{\gamma}$ submanifold of $R^{D}$. Let $C^{s}(\mathcal{M})$ be the set of all functions on $\mathcal{M}$ with h\"{o}lder smoothness $s$. Then for any $f\in C^{s}(\mathcal{M})$ with $s\leq \min\{2,\gamma\}$, there exist constants $a_0\geq1$, $C>0$ and $B>0$ depending only on $\mu$, $\mathcal{M}$ and $f$ such that for all $a\geq a_0$,
\[
\inf\{||h||_{\tilde{\mathbb{H}}^{a}}^2:\sup_{x\in \mathcal{M}}|h(x)-f(x)|\leq Ca^{-s}\}\leq Ba^d.
\]
\end{lemma}

\subsection{Centered Small Ball Probability}
As indicated by the proof of Lemma 4.6 in \cite{Van2009}, to obtain an upper bound on $-\log P(||W^a||_{\infty}<\epsilon)$, we need to provide an upper bound for the covering entropy $\log N(\epsilon,\tilde{\mathbb{H}}^{a}_1,||\cdot||_{\infty})$ of the unit ball in the RKHS $\tilde{\mathbb{H}}^{a}$ on the submanifold $\mathcal{M}$. Following the discussion in section 4.1, we want to link $\tilde{\mathbb{H}}^{a}$ to $\mathbb{H}^a$, the associated RKHS defined on the ambient space $\bbR^D$. Therefore, we need a lemma to characterize the space $\mathbb{H}^a$ \citep[Lemma 4.1]{Van2009}.

\begin{lemma}\label{le:2a}
$\mathbb{H}^{a}$ is the set of real parts of the functions
\[
 x\mapsto\int e^{i(\lambda,x)}\psi(\lambda)\mu_a(d\lambda),
 \]
 when $\psi$ runs through the complex Hilbert space $L_2(\mu_a)$. Moreover, the RKHS norm of the above function is $||\psi||_{L_2(\mu_a)}$, where $\mu_a$ is the spectral measure of the covariance function $K^a$.
\end{lemma}

Based on this representation of $\mathbb{H}^a$ on $\bbR^D$, \cite{Van2009} proved an upper bound $Ka^D\big(\log \frac{1}{\epsilon}\big)^{D+1}$ for $\log N(\epsilon,\tilde{\mathbb{H}}^{a}_1,||\cdot||_{\infty})$ through constructing an $\epsilon$-covering set composed of piecewise polynomials. However, there is no straightforward generalization of their scheme from Euclidean spaces to manifolds. The following lemma provides an upper bound for the covering entropy of $\tilde{\mathbb{H}}^{a}_1$, where the $D$ in the upper bounds for $\mathbb{H}^{a}_1$ is reduced to $d$. The main novelty in our proof is the construction of an $\epsilon$-covering set composed of piecewise transformed polynomials \eqref{eq:pp} via analytically extending the truncated Taylor polynomial approximations \eqref{eq:ru} of the elements in $\tilde{\mathbb{H}}^{a}_1$. As the proof indicates, the $d$ in $a^d$ relates to the covering dimension $d$ of $\mathcal{M}$, i.e. the $\epsilon$-covering number $N(\epsilon,\mathcal{M},\epsilon)$ of $\mathcal{M}$ is proportional to $1/\epsilon^d$. The $d$ in $(\log \frac{1}{\epsilon}\big)^{d+1}$ relates to the order of the number $k^d$ of coefficients in piecewise transformed polynomials of degree $k$ in $d$ variables.

\begin{lemma}\label{le:3a}
Assume that $\mathcal{M}$ is a $d$-dimensional $C^{\gamma}$ compact submanifold of $R^D$ with $\gamma\geq 2$. Then for squared exponential covariance function $K^a$, there exists a constant $K$ depending only on $d$, $D$ and $\mathcal{M}$, such that for $\epsilon<1/2$ and $a>\max\{a_0,\epsilon^{-1/(\gamma-1)}\}$, where $\delta_0$ is defined in Lemma 7.7 in the appendix and $a_0$ is a universal constant,
\[
\log N(\epsilon,\tilde{\mathbb{H}}^{a}_1,||\cdot||_{\infty})\leq Ka^d\bigg(\log \frac{1}{\epsilon}\bigg)^{d+1}.
\]
\end{lemma}

Similar to Lemma 4.6 in \cite{Van2009}, Lemma \ref{le:3a} implies an upper bound on $-\log P(||W^a||_{\infty}<\epsilon)$.

\begin{lemma}\label{le:4a}
Assume that $\mathcal{M}$ is a $d$-dimensional compact $C^{\gamma}$ submanifold of $\bbR^D$ with $\gamma\geq 2$. If $K^a$ is the squared exponential covariance function with inverse bandwidth $a$, then for some $a_0>0$, there exist constants $C$ and $\epsilon_0$ that only depend on $a_0$, $\mu$, $d$, $D$ and $\mathcal{M}$, such that, for $a\geq \max\{a_0, \epsilon^{-1/(\gamma-1)}\}$ and $\epsilon<\epsilon_0$,
\[
-\log P\big(\sup_{x\in \mathcal{M}}|W_x^a|\leq\epsilon\big)\leq Ca^d\bigg(\log \frac{a}{\epsilon}\bigg)^{d+1}.
\]
\end{lemma}

Before proving Theorem \ref{thm:main1}, we need another two technical lemmas for preparations, which are the analogues of Lemma 4.7 and 4.8 in \cite{Van2009} for RKHS on Euclidean spaces.
\begin{lemma}\label{le:5a}
For squared exponential covariance function, if $a\leq b$, then $\sqrt{a}\tilde{\mathbb{H}}^{a}_1\subset \sqrt{b}\tilde{\mathbb{H}}^{b}_1$.
\end{lemma}

\begin{lemma}\label{le:6a}
Any $h\in\tilde{\mathbb{H}}^{a}_1$ satisfies $|h(x)|\leq 1$ and $|h(x)-h(x')|\leq a||x-x'||\tau$ for any $x,x'\in\mathcal{M}$, where $\tau^2=\int||\lambda||^2d\mu(\lambda)$.
\end{lemma}

\subsection{Posterior Contraction Rate of GP on Manifold}
By using the manifold adapted lemmas in section 3.3 to 3.4, the proofs of Theorem \ref{thm:main1} and Corollary \ref{thm:main1b} follow similar ideas as the proof of Theorem 3.1 in \cite{Van2009} and are provided in the appendix.

\section{Proofs}\label{se:tp}
In this section, we provide technical proofs for the results in the paper.

\subsection{Proof of lemma \ref{le:1a}}
Consider the map $\Phi:\tilde{\mathcal{H}}\rightarrow \mathcal{H}$ that maps the function $$\sum_{i=1}^{m}a_iK^a(x_i,\cdot)\in \tilde{\mathcal{H}},\ a_1,\ldots,a_m\in \bbR, x_1,\ldots,x_m\in\mathcal{M}, m\in\mathds{N}$$ on $\mathcal{M}$ to
the function of the same form $$\sum_{i=1}^{m}a_iK^a(x_i,\cdot)\in\mathcal{H},$$ but viewed as a function on $[0,1]^D$. By definitions of RKHS norms, $\Phi$ is an isometry between $\tilde{\mathcal{H}}$ and a linear subspace of $\mathcal{H}$. As a result, $\Phi$ can be extended to an isometry between $\tilde{\mathbb{H}}^{a}$ and a complete subspace of $\mathbb{H}^{a}$. To prove the first part of this lemma, it suffices to justify that for any $f\in \tilde{\mathbb{H}}^{a}$, $g=\Phi(f)|_\mathcal{M}=f$. Assume that the sequence $\{f_n\}\in\tilde{\mathcal{H}}$ satisfies $$||f_n-f||_{\tilde{\mathbb{H}}^{a}}\rightarrow0, \text{ as } n\rightarrow\infty,$$ then by the definition of $\Phi$ on $\tilde{\mathcal{H}}$, $\Phi(f_n)|_\mathcal{M}=f_n$. For any $x\in [0,1]^D$, by the reproducing property and Cauchy-Schwarz inequality,
\begin{eqnarray*}
|\Phi(f_n)(x)-g(x)|&=&|\langle K^a(x,\cdot), \Phi(f_n)-g\rangle_{\mathbb{H}^{a}}|\\
&\leq&\sqrt{K^a(x,x)}\ ||\Phi(f_n)-\Phi(f)||_{\mathbb{H}^{a}}\\
&=&||f_n-f||_{\tilde{\mathbb{H}^{a}}}\rightarrow0, \text{ as }n\rightarrow\infty,
\end{eqnarray*}
where the last step is by isometry. This indicates that $g$ can be obtained as a point limit of $\Phi(f_n)$ on $[0,1]^D$ and in the special case when $x\in \mathcal{M}$, $$g(x)=\lim_{n\rightarrow\infty}\Phi(f_n)(x)=\lim_{n\rightarrow\infty}f_n(x)=f(x).$$
Denote the orthogonal complement of $\Phi(\tilde{\mathbb{H}}^{a})$ in $\mathbb{H}^{a}$ as $\Phi(\tilde{\mathbb{H}}^{a})^{\perp}$. Since $(g'-g)|_{\mathcal{M}}=0$, which means $\langle K^a(x,\cdot), g-g'\rangle_{\mathbb{H}^{a}}=0$ for all $x\in \mathcal{M}$. Therefore by the previous construction, $g-g'\perp \Phi(\tilde{\mathbb{H}}^{a})$, i.e. $g'-g\in\Phi(\tilde{\mathbb{H}}^{a})^{\perp}$ and using Pythagorean theorem, we have
\[
||g'||_{\mathbb{H}^{a}}^2=||g||_{\mathbb{H}^{a}}^2+||g-g'||_{\mathbb{H}^{a}}^2\geq ||g||^2_{\mathbb{H}^{a}}.
\]

\subsection{Proof of Lemma \ref{le:7}}
The proof consists of two parts. In the first part, we prove that the approximation error of $I_a(f)$ can be decomposed into four terms. The first term $T_1$ is the convolution error defined in our previous heuristic argument. The second term $T_2$ is caused by localization of the integration, which is negligible due to the exponential decaying of the squared exponential covariance function. The third and fourth terms $T_3$, $T_4$ correspond to the non-flat error, with $T_3$ caused by approximating the geodesic distance with Euclidean distance $\big|||\phi^q(u)-q||^2-||u||^2\big|$, and $T_4$ by approximating the Jacobian $\big|\sqrt{\det(g^{\phi}_{ij}(u))}-1\big|$. Therefore the overall approximation error $|I_a(f)(x)-f(x)|$ has order $s$ in the sense that for some constant $C>0$ dependent on $\mathcal{M}$ and $f$:
\begin{equation}\label{eq:7b}
\sup_{x\in \mathcal{M}}|I_a(f)(x)-f(x)|\leq Ca^{-s},\ s\leq\min\{2,\gamma\}.
\end{equation}
In the second part, we prove that $I_a(f)$ belongs to $\tilde{\mathbb{H}}^{a}$ and has a squared RKHS norm:
\[
||I_a(f)||_{\tilde{\mathbb{H}}^{a}}^2\leq Ba^d,
\]
where $B$ is a positive constant not dependent on $a$.

\emph{Step 1 (Estimation of the approximation error):} This part follows similar ideas as in the proof of Theorem 1 in \cite{Ye2008}, where they have shown that (\ref{eq:7b}) holds for $s\leq 1$. Our proof generalizes their results to $s\leq2$ and therefore needs more careful estimations.

By Proposition 7.5 in the appendix, for each $p\in \mathcal{M}$, there exists a neighborhood $W_p$ and an associated $\delta_p$ satisfying the two conditions in Proposition 7.4 and equations (7.4)-7(.6) in the appendix. By compactness, $\mathcal{M}$ can be covered by $\cup_{p\in\mathcal{P}}W_p$ for a finite subset $\mathcal{P}$ of $\mathcal{M}$. Then $\sup_{x\in \mathcal{M}}|I_a(f)(x)-f(x)|=\sup_{p\in\mathcal{P}}\{\sup_{x\in W_p}|I_a(f)(x)-f(x)|\}$.  Let $\delta^*=\min_{p\in\mathcal{P}}\{\min\{\delta_p,1/\sqrt{2C_p}\}\}>0$, where $C_p$ is defined as in equation (7.6) in the appendix. Choose $a_0\geq1$ sufficiently large such that $C_0\sqrt{(2d+8)\log a_0}/a_0<\delta^*$, where $C_0$ is the $C_2$ in Lemma 7.6 in the appendix.

Let $q\in W_p$ and $a>a_0$. Define $B_a^q=\big\{x\in \mathcal{M}:d_{\mathcal{M}}(q,x)<C_0$ $\sqrt{(2d+8)\log a}/a\big\}$. Combining equation (7.3) in the appendix and the fact that $\mathcal{E}_q$ is a diffeomorphism on $B_{\delta^*}(0)$,
\[
B_a^q=\big\{\mathcal{E}_q(\sum_{i=1}^du_ie_i^q):u\in\tilde{B}_a \big\}\subset\mathcal{E}_q(B_{\delta^*}(0)),
\]
where $\tilde{B}_a=\big\{u:||u||<C_0\sqrt{(2d+8)\log a}/a\big\}\subset B_{\delta^*}(0)$.

Denote $\phi^q(u)=\mathcal{E}_q(\sum_{i=1}^du_ie^q_i)$. Then $B^q_a=\phi^q(\tilde{B}_a)$. By definition (7.1) in the appendix,
\begin{align*}
\lefteqn{\bigg(\frac{a}{\sqrt{2\pi}}\bigg)^d\int_{B_a^q}K^a(q,y)f(y)dV(y)}\\
&=\bigg(\frac{a}{\sqrt{2\pi}}\bigg)^d\int_{\tilde{B}_a}\exp\bigg\{-\frac{a^2||q-\phi^q(u)||^2}{2}\bigg\}f(\phi^q(u))
    \sqrt{\det(g_{ij}^q)}(u)du.
\end{align*}
Therefore, by \eqref{eq:6b} we have the following decomposition:
\[
I_a(f)(q)-f(q)=T_1+T_2+T_3+T_4,
\]
where
\begin{align*}
T_1=&\bigg(\frac{a}{\sqrt{2\pi}}\bigg)^d\int_{\tilde{B}_a}\exp\bigg\{-\frac{a^2||u||^2}{2}\bigg\}\big[f\big(\phi^q(u)\big)-f\big(\phi^q(0)\big)\big]du\\
T_2=&\bigg(\frac{a}{\sqrt{2\pi}}\bigg)^d\int_{\mathcal{M}\backslash B_a^q}K^a(q,y)f(y)dV(y)\\
&\qquad\qquad\qquad\qquad-
\bigg(\frac{a}{\sqrt{2\pi}}\bigg)^d\int_{\bbR^d\backslash\tilde{B}_a}\exp\bigg\{-\frac{a^2||u||^2}{2}\bigg\}f(q)du,\\
T_3=&\bigg(\frac{a}{\sqrt{2\pi}}\bigg)^d\int_{\tilde{B}_a}\bigg\{\exp\bigg\{-\frac{a^2||q-\phi^q(u)||^2}{2}\bigg\}\\
&\qquad\qquad\qquad\qquad\qquad-\exp\bigg\{-\frac{a^2||u||^2}{2}\bigg\}\bigg\}f(\phi^q(u))du,\\
T_4=&\bigg(\frac{a}{\sqrt{2\pi}}\bigg)^d\int_{\tilde{B}_a}\exp\bigg\{-\frac{a^2||q-\phi^q(u)||^2}{2}\bigg\}f(\phi^q(u))
    (\sqrt{\det(g_{ij}^q)}(u)-1)du.
\end{align*}

\emph{Step 1.1 (Estimation of $T_1$):}
Let $g=f\circ \phi^q$. Since $f\in C^{s}(\mathcal{M})$ and $(\phi^q,B_{\delta^*}(0))$ is a $C^{\gamma}$ coordinate chart, we have $g\in C^{s}(\bbR^d)$ and therefore
\[
g(u)-g(0)=\left\{
            \begin{array}{lr}
               R(u,s), &  \text{if } 0<s\leq \min\{1,\gamma\},\\
               \sum_{i=1}^d\frac{\partial g}{\partial u_i}(0)u_i+R(u,s), &  \text{if } 1<s\leq\min\{2,\gamma\},
            \end{array}
          \right.
\]
where the remainder term $|R(u,s)|\leq C_1||u||^s$ for all $0<s\leq\min\{2,\gamma\}$.
Since $\tilde{B}_a$ is symmetric,
\[
\int_{\tilde{B}_a}\exp\bigg\{-\frac{a^2||u||^2}{2}\bigg\}u_idu=0,\quad i=1,\ldots,d,
\]
and therefore
\[
|T_1|\leq C_1\bigg(\frac{a}{\sqrt{2\pi}}\bigg)^d\int_{\tilde{B}_a}\exp\bigg\{-\frac{a^2||u||^2}{2}\bigg\}||u||^sdu=C_2a^{-s}.
\]

\emph{Step 1.2 (Estimation of $T_2$):}
Denote $T_2=S_1+S_2$ where $S_1$ and $S_2$ are the first term and second term of $T_2$, respectively. By Lemma 7.6 in the appendix, for $y\in\mathcal{M}\backslash B_a^q$, $||q-y||\geq d_{\mathcal{M}}(q,y)/C_0\geq \sqrt{(2d+8)\log a}/a$. Therefore,
\begin{align*}
|S_1|&=\bigg|\bigg(\frac{a}{\sqrt{2\pi}}\bigg)^d\int_{\mathcal{M}\backslash B_a^q}\exp\bigg\{-\frac{a^2||q-y||^2}{2}\bigg\}f(y)dV(y)\bigg|\\
    &\leq||f||_{\infty}\text{Vol}(\mathcal{M})\bigg(\frac{a}{\sqrt{2\pi}}\bigg)^d\exp\bigg\{-\frac{(2d+8)\log a}{2}\bigg\}\\
    &=C_3a^{-4}\leq C_3a^{-s}.
\end{align*}
As for $S_2$, we have
\begin{align*}
|S_2|&\leq||f||_{\infty} \bigg(\frac{a}{\sqrt{2\pi}}\bigg)^d\int_{||u||\geq C_0\sqrt{(2d+8)\log a}/a}\exp\bigg\{-\frac{a^2||u||^2}{2}\bigg\}du\\
&\leq||f||_{\infty} \bigg(\frac{a}{\sqrt{2\pi}}\bigg)^d\int_{\bbR^d}\exp\bigg\{-\frac{C_0^2(2d+8)\log a}{4}\bigg\}\exp\bigg\{-\frac{a^2||u||^2}{4}\bigg\}du\\
&=C_4a^{-C_0^2(d/2+2)}\leq C_4a^{-s},
\end{align*}
since $d\geq1$, $C_0\geq1$ and $a\geq a_0\geq1$.

Combining the above inequalities for $S_1$ and $S_2$, we obtain
\begin{align*}
&|T_2|\leq (C_3+C_4)a^{-s}=C_5a^{-s}.
\end{align*}

\emph{Step 1.3 (Estimation of $T_3$):}
By equation (7.6) in Proposition 7.5 and equation (7.3) in the appendix, we have
\begin{equation}\label{eq:8b}
\begin{aligned}
\big|||u||^2-||q-\phi^q(u)||^2\big|&=\big|d^2_{\mathcal{M}}(q,\phi^q(u))-||q-\phi^q(u)||^2\big|\\
&\leq C_pd^4_{\mathcal{M}}(q,\phi^q(u))=C_p||u||^4.
\end{aligned}
\end{equation}
Therefore by using the inequality $|e^{-a}-e^{-b}|\leq|a-b|\max\{e^{-a},e^{-b}\}$ for $a,b>0$, we have
\begin{align*}
|T_3|\leq||f||_{\infty}\bigg(\frac{a}{\sqrt{2\pi}}\bigg)^d&\int_{\tilde{B}_a}\max\bigg\{\exp\bigg\{-\frac{a^2||q-\phi^q(u)||^2}{2}\bigg\}
,\\
&\exp\bigg\{-\frac{a^2||u||^2}{2}\bigg\}\bigg\}\frac{a^2||u||^4}{2}du.
\end{align*}
By equation (\ref{eq:8b}) and the fact that $u\in \tilde{B}_a$, $||u||^2\leq(\delta^*)^2\leq1/(2C_p)$ and hence
\begin{align}
 &\big|||u||^2-||q-\phi^q(u)||^2\big|\leq \frac{1}{2}||u||^2,\quad ||q-\phi^q(u)||^2\geq \frac{1}{2}||u||^2.\label{eq:9b}
\end{align}
Therefore
\[
|T_3|\leq||f||_{\infty}\bigg(\frac{a}{\sqrt{2\pi}}\bigg)^d\int_{\tilde{B}_a}\exp\bigg\{-\frac{a^2||u||^2}{4}\bigg\}
\frac{a^2||u||^4}{2}du=C_6a^{-2}\leq C_6a^{-s},
\]
since $a\geq a_0\geq1$.

\emph{Step 1.4 (Estimation of $T_4$):}
By equation (7.5) in Proposition 7.5 in the appendix, there exists a constant $C_7$ depending on the Ricci tensor of the manifold $\mathcal{M}$, such that
\[
\big|\sqrt{\det(g_{ij}^q)}(u)-1\big|\leq C_7||u||^2.
\]
Therefore, by applying equation \eqref{eq:9b} again, we obtain
\[
|T_4|\leq C_4||f||_{\infty}\bigg(\frac{a}{\sqrt{2\pi}}\bigg)^d\int_{\tilde{B}_a}\exp\bigg\{-\frac{a^2||u||^2}{4}\bigg\}
||u||^2du=C_8a^{-2}\leq C_8a^{-s}.
\]

Combining the above estimates for $T_1$, $T_2$, $T_3$ and $T_4$, we have
\[
\sup_{x\in\mathcal{M}}|I_a(f)(q)(x)-f(q)(x)|\leq (C_2+C_3+C_6+C_8)a^{-s}=Ca^{-s}.
\]

\emph{Step 2 (Estimation of the RKHS norm):}
Since $\langle K^a(x,\cdot), K^a(y,\cdot)\rangle_{\tilde{\mathbb{H}}^{a}}=K^a(x,y)$, we have
\begin{align*}
||I_a(f)||_{\tilde{\mathbb{H}}^{a}}&=\bigg(\frac{a}{\sqrt{2\pi}}\bigg)^{2d}\int_{\mathcal{M}}\int_{\mathcal{M}}
K^a(x,y)f(x)f(y)dV(x)dV(y)\\
&\leq||f||_{\infty}^2\bigg(\frac{a}{\sqrt{2\pi}}\bigg)^{d}\int_{\mathcal{M}}dV(x)\bigg(\frac{a}{\sqrt{2\pi}}\bigg)^{d}\int_{\mathcal{M}}K^a(x,y)dV(y).
\end{align*}
Applying the results of the first part to function $f\equiv1$, we have
\[
\bigg|\bigg(\frac{a}{\sqrt{2\pi}}\bigg)^{d}\int_{\mathcal{M}}K^a(x,y)dV(y)-1\bigg|\leq Ca^{-2}\leq C,
\]
since $a\geq a_0\geq1$. Therefore,
\[
||I_a(f)||_{\tilde{\mathbb{H}}^{a}}\leq(1+C)||f||_{\infty}^2\bigg(\frac{a}{\sqrt{2\pi}}\bigg)^{d}\text{Vol}(\mathcal{M})=Ba^d.
\]

\subsection{Proof of Lemma \ref{le:3a}}
By Lemma \ref{le:1a} and Lemma \ref{le:2a}, a typical element of $\tilde{\mathbb{H}}^{a}$ can be written as the real part of the function
$$h_{\psi}(x)=\int e^{i(\lambda,x)}\psi(\lambda)\mu_a(d\lambda),\text{ for } x\in\mathcal{M}$$
for $\psi:\bbR^D\rightarrow\mathds{C}$ a function with $\int|\psi|^2\mu_a(d\lambda)\leq1$. This function can be extended to $\bbR^D$ by allowing $x\in\bbR^D$. For any given point $p\in\mathcal{M}$, by (7.2) in the appendix, we have a local coordinate $\phi^p:B_{\delta_0}(0)\subset \bbR^d\rightarrow \bbR^D$ induced by the exponential map $\mathcal{E}_p$. Therefore, for $x\in\phi_p(B_{\delta_0}(0))$, $h_{\psi}(x)$ can be written in local $q$-normal coordinates as
\begin{equation}\label{eq:rkhs}
    h_{\psi,p}(u)=h_{\psi}\big(\phi^p(u)\big)=
    \int e^{i(\lambda,\phi^p(u))}\psi(\lambda)\mu_a(d\lambda),\ u\in B_{\delta_0}(0).
\end{equation}

Similar to the idea in the proof of Lemma 4.5 in \cite{Van2009}, we want to extend the function $h_{\psi,p}$ to an analytical function $z\mapsto \int e^{i(\lambda,\phi^p(z))}\psi(\lambda)\mu_a(d\lambda)$ on the set $\Omega= \{z\in\mathds{C}^d:||\text{Re}z||<\delta_0,||\text{Im} z||<\rho/a\}$ for some $\rho>0$. Then we can obtain upper bounds on the mixed partial derivatives of the analytic extension $h_{\psi,p}$ via Cauchy formula, and finally construct an $\epsilon$-covering set of $\tilde{\mathbb{H}}^{a}_1$ by piecewise polynomials defined on $\mathcal{M}$.
Unfortunately, this analytical extension is impossible unless $\phi^p(u)$ is a polynomial. This motivates us to approximate $\phi^p(u)$ by its $\gamma$th order Taylor polynomial $P_{p,\gamma}(u)$. More specifically, by Lemma \ref{le:6a} and the discussion after Lemma 7.7 in the appendix, the error caused by approximating $\phi^p(u)$ by $P_{p,\gamma}(u)$ is
\begin{align}\label{eq:aerr}
    \big|h_{\psi}\big(\phi^p(u)\big)-h_{\psi}\big(P_{p,\gamma}(u)\big)\big|
    \leq a||\phi^p(u)-P_{p,\gamma}(u)||\leq Ca||u||^{\gamma}.
\end{align}
For notation simplicity, fix $p$ as a center and denote the function $h_{\psi}\big(P_{p,\gamma}(u)\big)$ by $r(u)$ for $u\in B_{\delta_0}$. Since $P_{p,\gamma}(u)$ is a polynomial of degree $\gamma$, view the function $r$ as a function of argument $u$ ranging over the product of the imaginary axes in $\mathds{C}^d$, we can extend
\begin{align}
    r(u)=\int e^{i(\lambda,P_{p,\gamma}(u))}\psi(\lambda)\mu_a(d\lambda), \ u\in B_{\delta_0}(0)\label{eq:ru}
\end{align}
to an analytical function $z\mapsto \int e^{i(\lambda,P_{p,\gamma}(z))}\psi(\lambda)\mu_a(d\lambda)$ on the set $\Omega= \{z\in\mathds{C}^d:||\text{Re}z||<\delta_0,||\text{Im} z||<\rho/a\}$ for some $\rho>0$ sufficiently small determined by the $\delta<1/2$ in (\ref{eq:1}). Moreover, by Cauchy-Schwarz inequality, $|r(z)|\leq C$ for $z\in\Omega$ and $C^2=\int e^{\delta||\lambda||}\mu(d\lambda)$. Therefore, by Cauchy formula, with $D^n$ denoting the partial derivative of orders $n=(n_1,\ldots,n_d)$ and $n!=n_1!\cdots n_d!$, we have the following bound for partial derivatives of $r$ at any $u\in B_{\delta_0}(0)$,
\begin{align}\label{eq:dub}
 \bigg|\frac{D^nr(u)}{n!}\bigg|\leq \frac{C}{R^n},
\end{align}
where $R=\rho/(a\sqrt{d})$. Based on the inequalities \eqref{eq:aerr} and \eqref{eq:dub}, we can construct an $\epsilon$-covering set of $\tilde{\mathbb{H}}^{a}_1$ as follows.

Set $a_0=\rho/(2\delta_0\sqrt{d})$, then $R<2\delta_0$. Since $\mathcal{M}\subset [0,1]^D$, with $C_2$ defined in Lemma 7.6 in the appendix, let $\{p_1,\ldots,p_m\}$ be an $R/(2C_2)$-net in $\mathcal{M}$ for the Euclidean distance, and let
$\mathcal{M}=\bigcup_iB_i$ be a partition of $\mathcal{M}$ in sets $B_1,\ldots,B_m$ obtaining by assigning every $x\in\mathcal{M}$ to the closest $p_i\in\{p_1,\ldots,p_m\}$. By (6.3) and Lemma 7.6 in the appendix
\begin{equation}\label{eq:difb}
    |(\phi^{p_i})^{-1}(x)|<C_2\frac{R}{2C_2}=\frac{R}{2}<\delta_0,
\end{equation}
where $\phi_{p_i}$ is the local normal coordinate chart at $p_i$. Therefore, we can consider the piecewise transformed polynomials $P=\sum_{i=1}^mP_{i,a_i}1_{B_i}$, with
\begin{align}
P_{i,a_i}(x)=\sum_{n_.\leq k}a_{i,n}[(\phi^{p_i})^{-1}(x)]^n, \ x\in \phi^{p_i}\big(B_{\delta_0}(0)\big).\label{eq:pp}
\end{align}
Here the sum ranges over all multi-index vectors $n=(n_1,\ldots,n_d)\in (\mathds{N}\cup\{0\})^d$ with $n_.=n_1+\cdots+n_d\leq k$. Moreover, for $y=(y_1,\ldots,y_d)\in \bbR^d$, the notation $y^n$ used above is short for $y_1^{n_1}y_2^{n_2}\cdots y_d^{n_d}$. We obtain a finite set of functions by discretizing the coefficients $a_{i,n}$ for each $i$ and $n$ over a grid of meshwidth $\epsilon/R^n$-net in the interval $[-C/R^n,C/R^n]$ (by \eqref{eq:dub}). The log cardinality of this set is bounded by
\[
\log \bigg(\prod_i\prod_{n:n_.\leq k}\#a_{i,n}\bigg)\leq m\log\bigg(\prod_{n:n_.\leq k}\frac{2C/R^n}{\epsilon/R^n}\bigg)\leq mk^d\log\bigg(\frac{2C}{\epsilon}\bigg).
\]
Since $R=\rho/(a\sqrt{d})$, we can choose $m=N\big(\mathcal{M},||\cdot||,\rho/(2C_0ad^{1/2})\big)\simeq a^d$. To complete the proof, it suffices to show that for $k$ of order $\log(1/\epsilon)$, the resulting set of functions is a $K\epsilon$-net for constant $K$ depending only on $\mu$.

For any function $f\in\tilde{\mathbb{H}}^a_1$, by Lemma \ref{le:1a}, we can find a $g\in \tilde{\mathbb{H}}^a_1$ such that $g|_{\mathcal{M}}=f$.
Assume that $r_g$ (the subcript $g$ indicates the dependence on $g$) is the local polynomial approximation for $g$ defined as \eqref{eq:ru}. Then we have a partial derivative bound on $r_g$ as:
\[
\bigg|\frac{D^nr_g(p_i)}{n!}\bigg|\leq \frac{C}{R^n}.
\]
Therefore there exists a universal constant $K$ and appropriately chosen $a_i$ in \eqref{eq:pp}, such that for any $z\in B_i\subset \mathcal{M}$,
\[
\bigg|\sum_{n_.>k}\frac{D^nr_g(p_i)}{n!}(z-p_i)^n\bigg|\leq\sum_{n_.>k}\frac{C}{R^n}(R/2)^n\leq C\sum_{l=k+1}^{\infty}\frac{l^{d-1}}{2^l}\leq KC\bigg(\frac{2}{3}\bigg)^k,
\]
\[
\bigg|\sum_{n_.\leq k}\frac{D^nr_g(p_i)}{n!}(z-p_i)^n-P_{i,n_i}(z)\bigg|\leq\sum_{n_.\leq k}\frac{\epsilon}{R^n}(R/2)^n\leq\sum_{l=1}^{k}\frac{l^{d-1}}{2^l}\epsilon\leq K\epsilon.
\]
Moreover, by \eqref{eq:aerr} and \eqref{eq:difb},
\[
|g(z)-r_g(z)|\leq Ca||(\phi^{p_i})^{-1}(z)||^{\gamma}\leq a R^{\gamma}\leq K a^{-(\gamma-1)}<K\epsilon,
\]
where the last step follows by the condition on $a$.

Consequently, we obtain
\begin{align*}
  |f(z)-P_{i,n_i}(z)|=&|g(z)-P_{i,n_i}(z)|\leq |g(z)-r_g(z)|+|r_g(z)-P_{i,n_i}(z)|\\
  \leq &KC\bigg(\frac{2}{3}\bigg)^k+2K\epsilon.
\end{align*}
This suggests that the piecewise polynomials form a $3K\epsilon$-net for $k$ sufficiently large so that $(2/3)^k$ is smaller than $K\epsilon$.

\subsection{Proof of Lemma \ref{le:5a}}
For any $f\in \sqrt{a}\tilde{\mathbb{H}}^{a}_1$, by Lemma \ref{le:1a}, there exists $g\in \sqrt{a}\mathbb{H}^{a}_1$ such that $g|_{\mathcal{M}}=f$. By Lemma 4.7 in
\cite{Van2009}, $\sqrt{a}\mathbb{H}^{a}_1\subset \sqrt{b}\mathbb{H}^{b}_1$, so $g\in \sqrt{b}\mathbb{H}^{b}_1$. Again by Lemma \ref{le:1a}, since $g|_{\mathcal{M}}=f$, $||f||_{\tilde{\mathbb{H}}^{b}}\leq||g||_{\mathbb{H}^{b}}\leq\sqrt{b}$, implying that $f\in \sqrt{b}\tilde{\mathbb{H}}^{b}_1$.

\subsection{Proof of Lemma \ref{le:6a}}
By the reproducing property and Cauchy-Schwarz inequality
\begin{align*}
 |h(x)|&=|\langle h, K^a(x,\cdot)\rangle_{\tilde{\mathbb{H}}^{a}}|\leq||K^a(x,\cdot)||_{\tilde{\mathbb{H}}^{a}}=1\\
 |h(x)-h(x')|&=|\langle h, K^a(x,\cdot)-K^a(x',\cdot)\rangle_{\tilde{\mathbb{H}}^{a}}|\\
 &\leq ||K^a(x,\cdot)-K^a(x',\cdot)||_{\tilde{\mathbb{H}}^{a}}\\
 &=\sqrt{2(1-K^a(x,x'))}.
\end{align*}
By the spectral representation $K(x,x')=\int e^{i(\lambda,t)}\mu_a(d\lambda)$ and the fact that $\mu_a$ is symmetric,
\begin{align*}
2(1-K^a(x,x'))&=2\int(1+i(\lambda,x-x')-e^{i(\lambda,x-x')})\mu_a(d\lambda)\\
&\leq ||x-x'||^2\int||\lambda||^2\mu_a(d\lambda)\\
&=a^2||x-x'||^2\int||\lambda||^2\mu(d\lambda).
\end{align*}

\subsection{Proof of Theorem \ref{le:BEr}}
First, we consider a fixed-designed case.
According to the proof of Theorem \eqref{thm:main1}, for any $t\geq 1$, there exists a sequence of sieves $\{B_{t,n}:n\geq1\}$ such that
\begin{align*}
    P(W^A\notin B_{t,n})\leq e^{-4n\epsilon_n^2t^2}\text{ and }
    \log N(t\bar{\epsilon}_n,B_{t,n},\|\cdot\|_{\infty})\leq n\bar{\epsilon}_n^2t^2.
\end{align*}
As a consequence, by borrowing some results in the proof of Theorem 2.1 in \cite{Ghosal2000}, it can be shown that there exists a sequence of measurable sets $A_n$ with $P_0(A_n^c)\to 0$, such that for some positive constant $c$ and any $t\geq1$,
\begin{align*}
    E_0 \ I(A_n)\ \Pi(\|f-f_0\|_n\geq t\bar{\epsilon}_n|S_n)\leq e^{-cn\epsilon_n^2t^2}.
\end{align*}
By plugging in $t=1,2,\ldots$ into the above display, dividing both sides by $\exp\{-cn\epsilon_n^2t^2/2\}$ and taking a summation, we can obtain
\begin{align*}
    E_0 \ I(A_n)\sum_{k=1}^{\infty}\Pi(\|f-f_0\|_n\geq k\bar{\epsilon}_n|S_n)\ e^{cn\epsilon_n^2k^2/2}\leq \sum_{k=1}^{\infty}e^{-cn\epsilon_n^2k^2/2}\to 0,
\end{align*}
as $n\to\infty$. As a consequence, there exists a sequence of sets $\{S_n\}$ with $P_0(S_n)\to 1$ as $n\to \infty$, such that for any $S_n\in S_n$ the following inequality holds uniformly for all $t\geq 1$ for some constant $c_0>0$:
\begin{align}\label{eq:eb}
    \Pi(\|f-f_0\|_n\geq t\bar{\epsilon}_n|S_n)\leq e^{-c_0n\epsilon_n^2t^2}.
\end{align}
Then for any $S_n\in S_n$, by Fubini's theorem we have
\begin{align*}
    \int \|f-f_0\|_n^2d\Pi(f|S_n)=&\ \int_{0}^{\infty}\Pi(\|f-f_0\|_n^2\geq s|S_n)ds\\
    \leq&\ \bar{\epsilon}_n^2+\int_{\bar{\epsilon}_n^2}^{\infty} \Pi(\|f-f_0\|_n^2\geq s|S_n)ds\\
    \leq&\ \bar{\epsilon}_n^2+\int_{\epsilon_n^2}^{\infty}e^{-c_0ns}ds\lesssim \bar{\epsilon}_n^2.
\end{align*}

Since $\hat{f}=\int f_Ad\Pi(f|S_n)$, we have decomposition $\int \|f_A-f_0\|_n^2d\Pi(f|S_n)=\int \|f_A-\hat{f}\|_n^2d\Pi(f|S_n)+\|\hat{f}-f_0\|_n^2$. Combining this decomposition with the fact that $|f_A(x)-f_0(x)|\leq |f(x)-f_0(x)|$ for any $x$, we have
\begin{align*}
    \|\hat{f}-f_0\|_n^2\leq \int \|f-f_0\|_n^2d\Pi(f|S_n).
\end{align*}
By the preceding displays, we can conclude that for all $S_n\in S_n$, $\|\hat{f}-f_0\|_n^2\lesssim \bar{\epsilon}_n^2$, which completes the proof for the fixed designed case.

The proof for a random-designed case is more involved. We will utilize the following result for comparing $\|\cdot\|_n$ and $\|\cdot\|_2$ based on empirical process theory \citep[Lemma 5.16]{geer2000}. Let $H_B(\epsilon,\mathcal{F},\|\cdot\|)$ denote the $\epsilon$-bracketing entropy of a function space $\mathcal{F}$ with respect to a norm $\|\cdot\|$.

\begin{lemma}\label{le:tnc}
Suppose $\sup_{f\in\mathcal{F}}\|f\|_{\infty}\leq A$.
For any $\delta$ satisfying $n\delta^2\geq H_B(\delta,\mathcal{F},\|\cdot\|_2)$ and $\eta\in(0,1)$, we have
\begin{align*}
    P_0\bigg(\sup_{f\in\mathcal{F},\ \|f\|_2\geq 2^5\delta/\eta}\bigg|
    \frac{\|f\|_n}{\|f\|_2}-1\bigg|\geq \eta\bigg)\leq 8\exp(-Cn\delta^2\eta^2),
\end{align*}
where the constant $C>0$ only depends on $A$.
\end{lemma}

According to Lemma 1 of \cite{Ghosal2007} and the proof of Theorem \eqref{thm:main1}, for all $S_n\in S_n$,
\begin{align*}
    \Pi(f\notin B_n|S_n)\leq e^{-cn\epsilon_n^2}\text{ and }\log N(3\bar{\epsilon}_n,B_n,\|\cdot\|_{\infty})\leq n\bar{\epsilon}_n^2.
\end{align*}
where $B_n$ is defined at the end of its proof.
Let $B_{n,A}=\{f_A:f\in B_n\}$. If $\{f^{(j)}\}$ forms an $\epsilon$-net of $B_n$, then $\{f^{(j)}_A\}$ forms an $\epsilon$-net of $B_{n,A}$. As a result, the covering entropy of $B_{n,A}$ is bounded by that of $B_n$. Combining this, \eqref{eq:eb} and the fact that an $\epsilon$-bracket entropy is always bounded by an $\epsilon$-covering entropy with respect to $\|\cdot\|_{\infty}$, we have for all $S_n\in S_n$,
\begin{align*}
    \Pi(\|f_A-&f_0\|_n\leq \bar{\epsilon}_n,f_A\in B_{n,A}|S_n)\geq 1-2e^{-cn\epsilon_n^2},\\
     &H_B(3\bar{\epsilon}_n,B_{n,A},\|\cdot\|_2)\leq n\bar{\epsilon}_n^2.
\end{align*}
Applying Lemma \ref{le:tnc} for $\mathcal{F}=B_{n,A}-f_0$, $\delta=\bar{\epsilon}_n$ and $\eta=1/2$, we have
a set $E_n$ with $P_0(E_n)\to 1$ as $n\to\infty$ such that for all $S_n\in E_n$,
\begin{align*}
    \frac{1}{2}\leq\sup_{f_A\in B_{n,A},\ \|f_A-f_0\|_2\geq 64\bar{\epsilon}_n}
    \frac{\|f_A-f_0\|_n}{\|f_A-f_0\|_2}\leq \frac{3}{2}.
\end{align*}
As a consequence, for all $S_n\in C_n\cap S_n$, we have
$\Pi(\|f_A-f_0\|_2\leq 64\bar{\epsilon}_n|S_n)\geq 1-2e^{-cn\epsilon_n^2}$ and
\begin{align*}
    \int \|f_A-f_0\|_2^2d\Pi(f|S_n)\leq 64^2\bar{\epsilon}_n^2+4A^2\Pi(\|f_A-f_0\|_2\geq 64\bar{\epsilon}_n|S_n)\lesssim \bar{\epsilon}_n^2.
\end{align*}
Therefore, we have $\|\hat{f}-f_0\|_2^2\leq \int \|f_A-f_0\|_2^2d\Pi(f|S_n)\lesssim \bar{\epsilon}_n^2$.

\subsubsection{Proof of Theorem \ref{thm:crv}}
The proof consists of two parts. In the first part, we show that for any $k=1,\ldots,d_{max}$, the associated MSPE $E_m^{(k)}=m^{-1}\sum_{i=1}^m\big(\hat{f}^{(k)}(\tilde{X}_i)-\tilde{Y}_i\big)^2$ on $\tilde{S}_m$ is a good estimator of $\|\hat{f}_k-f_0\|_2^2$ up to some fixed additive constant. The proof is an application of Bernstein's inequality. In the second part, we show that the estimator $\hat{f}_{CV}$ selected by cross validation can achieve an optimal convergence rate that is adaptive to the unknown dimensionality $d$. The proof borrows some results in the proof of Theorem \eqref{thm:main1}.

Step one: Let $\|\cdot\|_m$ denote the empirical $L_2$-norm on the testing set.
Since $\tilde{Y}_i=f_0(\tilde{X}_i)+\tilde{\epsilon}_i$, we can expand $E_m^{(k)}$ as
\begin{align}\label{eq:En}
    E_m^{(k)}=\|\hat{f}^{(k)}-f_0\|_m^2-\frac{2}{m}\sum_{i=1}^m\tilde{\epsilon}_i
    \big(\hat{f}^{(k)}(\tilde{X}_i)-f_0(\tilde{X}_i)\big)+\frac{1}{m}\sum_{i=1}^m\tilde{\epsilon}_i^2.
\end{align}
Since $\tilde{\epsilon}_i\overset{iid}{\sim} N(0,\sigma^2)$ are independent of $\hat{f}^{(k)}$ and $\tilde{X}_i$, the second term has a conditional distribution as normal with mean zero and variance $4\|\hat{f}^{(k)}-f_0\|_m^2/m$ conditioning on $\hat{f}^{(k)}$ and $\{\tilde{X}_i\}$. Therefore, we have for any $t>0$,
\begin{align}\label{eq:Gt}
    P_0\bigg(\bigg|\frac{2}{m}\sum_{i=1}^m\tilde{\epsilon}_i
    \big(\hat{f}^{(k)}(\tilde{X}_i)-f_0(\tilde{X}_i)\big)\bigg|\geq \|\hat{f}^{(k)}-f_0\|_m\frac{t}{\sqrt{m}}\bigg)\leq 2e^{-\frac{1}{8}t^2}.
\end{align}
Since $\hat{f}^{(k)}$ is independent of $\{\tilde{X}_i\}$ and $\|\hat{f}^{(k)}-f_0\|_{\infty}\leq 2A$, an application of Bernstein's inequality yields
\begin{align*}
    P_0\big(\big|\|\hat{f}^{(k)}-f_0\|_m^2-\|\hat{f}^{(k)}-f_0\|_2^2\big|\geq t\big)
    \leq 2\exp\bigg\{-\frac{mt^2}{8(Kt+\|\hat{f}^{(k)}-f_0\|_2^2)}\bigg\},
\end{align*}
for all $t>0$, where $K>0$ is a constant that only depends on $A$.
By choosing $t=\frac{1}{2}\|\hat{f}^{(k)}-f_0\|_2^2$, we have for some constant $C>0$,
\begin{align}\label{eq:ber}
    P_0\bigg(\frac{1}{2}\leq\frac{\|\hat{f}^{(k)}-f_0\|_m^2}{\|\hat{f}^{(k)}-f_0\|_2^2}
    \leq \frac{3}{2}\bigg)\geq 1-2\exp(-Cm\|\hat{f}^{(k)}-f_0\|_2^2).
\end{align}

By \eqref{eq:En}, \eqref{eq:Gt} and \eqref{eq:ber}, we have
\begin{align*}
    P_0\bigg(\frac{1}{2}\|\hat{f}^{(k)}-f_0\|_2^2+&\ \frac{1}{m}\sum_{i=1}^m\tilde{\epsilon}_i^2
    \leq E_m^{(k)} + \|\hat{f}^{(k)}-f_0\|_m\frac{t}{\sqrt{m}}\bigg)\\
    &\geq 1-
    2\exp\bigg\{-\frac{t^2}{8}\bigg\}-2\exp(-Cm\|\hat{f}^{(k)}-f_0\|_2^2),\\
    P_0\bigg(\frac{3}{2}\|\hat{f}^{(k)}-f_0\|_2^2+&\ \frac{1}{m}\sum_{i=1}^m\tilde{\epsilon}_i^2
    \geq E_m^{(k)} - \|\hat{f}^{(k)}-f_0\|_m\frac{t}{\sqrt{m}}\bigg)\\
    &\geq 1-
    2\exp\bigg\{-\frac{t^2}{8}\bigg\}-2\exp(-Cm\|\hat{f}^{(k)}-f_0\|_2^2).
\end{align*}
By choosing $t=\frac{1}{4}\sqrt{m}\|\hat{f}^{(k)}-f_0\|_2$ in the first inequality and $t=\frac{1}{2}\sqrt{m}\|\hat{f}^{(k)}-f_0\|_2$ in the second, we obtain the following key inequality
\begin{equation}\label{eq:cvkey}
\begin{aligned}
    P_0\bigg(\frac{1}{4}&\|\hat{f}^{(k)}-f_0\|_2^2+\frac{1}{m}\sum_{i=1}^m\tilde{\epsilon}_i^2
    \leq E_m^{(k)} \leq
    2\|\hat{f}^{(k)}-f_0\|_2^2+\frac{1}{m}\sum_{i=1}^m\tilde{\epsilon}_i^2\bigg)\\
    &\geq 1-8\exp(-Cm\|\hat{f}^{(k)}-f_0\|_2^2)\to 1\text{ as }n\to\infty.
\end{aligned}
\end{equation}

Step two: By the selection rule for $\hat{f}_{CV}=\hat{f}^{(\hat{d}_{CV})}$ and the assumption on $m$, we have that with probability tending to one,
\begin{align*}
    \frac{1}{4}\|\hat{f}_{CV}-f_0\|_2^2+\frac{1}{m}\sum_{i=1}^m\tilde{\epsilon}_i^2
    \leq E_m^{(\hat{d}_{CV})}
    \leq E_m^{(d)}\leq 2\|\hat{f}^{(d)}-f_0\|_2^2+\frac{1}{m}\sum_{i=1}^m\tilde{\epsilon}_i^2.
\end{align*}

By the preceding display and Theorem \ref{le:BEr} with the true dimension $d$, we can conclude that with probability tending to one,
\begin{align*}
    \|\hat{f}_{CV}-f_0\|_2\leq 2\sqrt{2}\|\hat{f}^{(d)}-f_0\|_2\lesssim \bar{\epsilon}_n\lesssim n^{-s/(2s+d)}(\log n)^{d+1}.
\end{align*}

\vskip 1em \centerline{\Large \bf Appendix} \vskip 1em

\section{Geometric Properties}


We introduce some concepts and results in differential and Riemannian geometry, which play an important role in the convergence rate. For detailed definitions and notations, the reader is referred to \cite{doCarmo1992}.

\subsection{Riemannian Manifold}
A manifold is a topological space that locally resembles Euclidean space. A $d$-dimensional topological manifold $\mathcal{M}$ can be described using an atlas, where an atlas is defined as a collection $\{(U_{s},\phi_{s})\}$ such that $\mathcal{M}=\bigcup_s U_{s}=$ and each chart $\phi_{s}:V\rightarrow U_{s}$ is a homeomorphism from an open subset $V$ of $d$-dimensional Euclidean space to an open subset $U_{s}$ of $\mathcal{M}$. By constructing an atlas whose transition functions $\{\tau_{s,\beta}=\phi_{\beta}^{-1}\circ\phi_{s}\}$ are $C^{\gamma}$ differentiable, we can further introduce a differentiable structure on $\mathcal{M}$. With this differentiable structure, we are able to define differentiable functions and their smoothness level $s\leq \gamma$. Moreover, this additional structure allows us to extend Euclidean differential calculus to the manifold. To measure distances and angles on a manifold, the notion of Riemannian manifold is introduced. A Riemannian manifold $(\mathcal{M},g)$ is a differentiable manifold $\mathcal{M}$ in which each tangent space $T_p\mathcal{M}$ is equipped with an inner product $\langle\cdot,\cdot\rangle_p=g_p(\cdot,\cdot)$ that varies smoothly in $p$.
The family $g_p$ of inner products is called a Riemannian metric and is denoted by $g$. With this Riemannian metric $g$, a distance $d_{\mathcal{M}}(p,q)$ between any two points $p,q\in\mathcal{M}$ can be defined as the length of the shortest path on $\mathcal{M}$ connecting them. For a given manifold $\mathcal{M}$, such as the set $P(n)$ of all $n\times n$ positive symmetric matrices \citep{Moakher2011,Hiai2009}, a Riemannian metric $g$ is not uniquely determined and can be constructed in various manners so that certain desirable properties, such as transformation or group action invariability, are valid. Although $g$ is not uniquely determined, the smoothness of a given function $f$ on $\mathcal{M}$ only depends on $\mathcal{M}$'s differential structure instead of its Riemannian metric. Therefore, to study functions on the manifold $\mathcal{M}$, we could endow it with any valid Riemannian metric. Since a low dimensional manifold structure on the $\bbR^D$-valued predictor $X$ is assumed in this paper, we will focus on the case in which $\mathcal{M}$ is a submanifold of a Euclidean space.
\begin{definition}\label{def:1}
 $\mathcal{M}$ is called a $C^{\gamma}$ submanifold of $\bbR^D$ if there exists an inclusion map $\Phi:\mathcal{M}\mapsto \bbR^D$, called embedding, such that $\Phi$ is a diffeomorphism between $\mathcal{M}$ and  $\Phi(\mathcal{M})\subset\bbR^D$, which means:
 \begin{enumerate}
   \item[(1)] $\Phi$ is injective and $\gamma$-differentiable;
   \item[(2)] The inverse $\Phi^{-1}:\Phi(\mathcal{M})\rightarrow \mathcal{M}$ is also $\gamma$-differentiable.
 \end{enumerate}
 \end{definition}
A natural choice of the Riemannian metric $g$ of $\mathcal{M}$ is the one induced by the Euclidean metric $e$ of $\bbR^D$ through
\[
g_p(u,v)=e_{\Phi(p)}(d\Phi_p(u),d\Phi_p(v))=\langle d\Phi_p(u),d\Phi_p(v)\rangle_{\bbR^D},\quad \forall u,v\in T_p\mathcal{M},
\]
for any $p\in \mathcal{M}$. Under this Riemannian metric $g$, $d\Phi_p:T_p\mathcal{M}\mapsto d\Phi_p(T_p\mathcal{M})\subset T_{\Phi(p)}\bbR^D$ is an isometric embedding. Nash Embedding Theorem \citep{Nash1956} implies that any valid Riemannian metric on $\mathcal{M}$ could be considered as being induced from a Euclidean metric of $\bbR^m$ with a sufficiently large $m$. Therefore, we would use this naturally induced $g$ as the Riemannian metric of predictor manifold $\mathcal{M}$ when studying the posterior contraction rate of our proposed GP prior defined on this manifold. Under such choice of $g$, $\mathcal{M}$ is isometrically embedded in the ambient space $\bbR^D$. In addition, in the rest of this paper, we will occasionally identify $\mathcal{M}$ with $\Phi(\mathcal{M})$ when no confusion arises.

Tangent spaces and Riemannian metric can be represented in terms of local parameterizations. Let $\phi: U\mapsto \mathcal{M}$ be a chart that maps a neighborhood $U$ of the origin in $\bbR^d$ to a neighborhood $\phi(U)$ of $p\in\mathcal{M}$. In the case that $\mathcal{M}$ is a $C^{\gamma}$ submanifold of $\bbR^D$, $\phi$ itself is $\gamma$-differentiable as a function from $U\in\bbR^d$ to $\bbR^D$. Given $i\in\{1,\ldots,d\}$ and $q=\phi(u)$, where $u=(u_1,\ldots,u_d)\in U$, define $\frac{\partial}{\partial u_i}(q)$ to be the linear functional on $C^{\gamma}(\mathcal{M})$ such that
\[
\frac{\partial}{\partial u_i}(q)(f)=\frac{d(f\circ \phi(u+te_i))}{dt}\bigg|_{t=0},\ \forall f\in C^{\gamma}(\mathcal{M}),
\]
where the $d$-dimensional vector $e_i$ has $1$ in the $i$-th component and $0$'s in others. Then $\frac{\partial}{\partial u_i}(q)$ can be viewed as a tangent vector in the tangent space $T_q\mathcal{M}$. Moreover, $\{\frac{\partial}{\partial u_i}(q):1\leq i\leq d\}$ forms a basis of $T_q\mathcal{M}$ so that each tangent vector $v\in T_q\mathcal{M}$ can written as
\[
v=\sum_{i=1}^d v_i\frac{\partial}{\partial u_i}(q).
\]
Under this basis, the tangent space of $\mathcal{M}$ can be identified as $\bbR^d$ and the matrix representation of differential $d\Phi_q$ at $q$ has a $(j,i)$th element given by
\[
\bigg\{d\Phi_q\bigg(\frac{\partial}{\partial u_i}\bigg)\bigg\}_j=\frac{d(\Phi_j\circ \phi(u+te_i))}{dt}\bigg|_{t=0},\ i=1,\ldots,d,\ j=1,\ldots,D,
\]
where we use the notation $F_j$ to denote the $j$th component of a vector-valued function $F$.
In addition, under the same basis, the Riemannian metric $g_q$ at $q$ can be expressed as
\[
g_q(v,w)=\sum_{i,j=1}^dv_iw_jg^{\phi}_{ij}(u_1,\ldots,u_d),
\]
where $(v_1,\ldots,v_d)$ and $(w_1,\ldots,w_d)$ are the local coordinates for $v,w\in T_q\mathcal{M}$. By the isometry assumption,
\[
g^{\phi}_{ij}(u_1,\ldots,u_d)=\langle d\Phi_q(\frac{\partial}{\partial u_i}),d\Phi_q(\frac{\partial}{\partial u_j})\rangle_{R^D}.
\]

Riemannian volume measure (form) of a region $R$ contained in a coordinate neighborhood $\phi(U)$ is defined as
\[
\text{Vol}(R)=\int_{R}dV(q)\triangleq\int_{\phi^{-1}(R)}\sqrt{\det(g^{\phi}_{ij}(u))}du_1\ldots du_d.
\]
The volume of a general compact region $R$, which is not contained in a coordinate neighborhood, can be defined through partition of unity \citep{doCarmo1992}. Vol generalizes the Lebesque measure of Euclidean spaces and can be used to define the integral of a function $f\in C(\mathcal{M})$ as
$\int_{\mathcal{M}}f(q)dV(q)$. In the special case that $f$ is supported on a coordinate neighborhood $\phi(U)$,
\begin{equation}\label{eq:0b}
\int_{\mathcal{M}}f(q)dV(q)=\int_{U}f(\phi(u))\sqrt{\det(g^{\phi}_{ij}(u))}du_1\ldots du_d.
\end{equation}

\subsection{Exponential Map}
Geodesic curves, generalizations of straight lines from Euclidean spaces to curved spaces, are defined as those curves whose tangent vectors remain parallel if they are transported and are locally the shortest path between points on the manifold. Formally, for $p\in\mathcal{M}$ and $v\in T_p\mathcal{M}$, the geodesic $\gamma(t,p,v),t>0$, starting at $p$ with velocity $v$, i.e. $\gamma(0,p,v)=p$ and $\gamma'(t,p,v)=v$, can be found as the unique solution of an ordinary differential equation. The exponential map $\mathcal{E}_p:T_p\mathcal{M}\mapsto\mathcal{M}$ is defined by $\mathcal{E}_p(v)=\gamma(1,p,v)$ for any $v\in T_p\mathcal{M}$ and $p\in\mathcal{M}$. Under this special local parameterization, calculations can be considerably simplified since quantities such as $\mathcal{E}_p$'s differential and Riemannian metric would have simple forms.

Although Hopf-Rinow theorem ensures that for compact manifolds the exponential map $\mathcal{E}_p$ at any point $p$ can be defined on the entire tangent space $T_p\mathcal{M}$, generally this map is no longer a global diffeomorphism. Therefore to ensure good properties of this exponential map, the notion of a normal neighborhood is introduced as follows.

\begin{definition}
A neighborhood $V$ of $p\in \mathcal{M}$ is called normal if:
\begin{enumerate}
  \item[(1)] Every point $q\in V$ can be joined to $p$ by a unique geodesic $\gamma(t,p,v),0\leq t\leq1$, with $\gamma(0,p,v)=p$ and $\gamma(1,p,v)=q$;
  \item[(2)] $\mathcal{E}_p$ is a diffeomorphism between $V$ and a neighborhood of the origin in $T_p\mathcal{M}$.
\end{enumerate}
\end{definition}

Proposition 2.7 and 3.6 in \cite{doCarmo1992} ensure that every point in $\mathcal{M}$ has a normal neighborhood. However, if we want to study some properties that hold uniformly for all exponential maps $\mathcal{E}_q$ with $q$ in a small neighborhood of $p$, we need a notion stronger than normal neighborhood, whose existence has been established in Theorem 3.7 in \cite{doCarmo1992}.

\begin{definition}
A neighborhood $W$ of $p\in \mathcal{M}$ is called uniformly normal if there exists some $\delta>0$ such that:
\begin{enumerate}
  \item[(1)] For every $q\in W$, $\mathcal{E}_p$ is defined on the $\delta$-ball $B_{\delta}(0)\subset T_q\mathcal{M}$ around the origin of $T_q\mathcal{M}$. Moreover, $\mathcal{E}_p(B_{\delta}(0))$ is a normal neighborhood of $q$;
  \item[(2)] $W\subset \mathcal{E}_p(B_{\delta}(0))$, which implies that $W$ is a normal neighborhood of all its points.
\end{enumerate}
\end{definition}

Moreover, as pointed out by \cite{Gine2005} and \cite{Ye2008}, by shrinking $W$ and reducing $\delta$ at the same time, a special uniformly normal neighborhood can be chosen.

\begin{proposition}\label{prop:1}
For every $p\in\mathcal{M}$ there exists a neighborhood $W$ such that:
\begin{enumerate}
  \item[(1)] $W$ is a uniformly normal neighborhood of $p$ with some $\delta>0$;
  \item[(2)] The closure of $W$ is contained in a strongly convex neighborhood $U$ of $p$;
  \item[(3)] The function $F(q,v)=(q,\mathcal{E}_q(v))$ is a diffeomorphism from $W_{\delta}=W\times B_{\delta}(0)$ onto its image in $\mathcal{M}\times\mathcal{M}$. Moreover, $|dF|$ is bounded away from zero on $W_{\delta}$.
\end{enumerate}
Here $U$ is strongly convex if for every two points in $U$, the minimizing geodesic joining them also lies in $U$.
\end{proposition}

Throughout the rest of the paper, we will assume that the uniformly normal neighborhoods also possess the properties in the above proposition. Given a point $p\in \mathcal{M}$, we choose a uniformly normal neighborhood $W$ of $p$. Let $\{e_1,\ldots,e_d\}$ be an orthonormal basis of $T_p\mathcal{M}$. For each $q\in W$, we can define a set of tangent vectors $\{e_1^q,\ldots,e_d^q\}\subset T_q\mathcal{M}$ by parallel transport \citep{doCarmo1992}: $e_i\in T_p\mathcal{M}\mapsto e_i^{\gamma(t)}\in T_{\gamma(t)}\mathcal{M}$ from $p$ to $q$ along the unique minimizing geodesic $\gamma(t)\ (0\leq t\leq 1)$ with $\gamma(0)=p, \gamma(1)=q$. Since parallel transport preserves the inner product in the sense that $g_{\gamma(t)}(v^{\gamma(t)},w^{\gamma(t)})=g_p(v,w),\forall v,w\in T_p\mathcal{M}$, $\{e_1^q,\ldots,e_d^q\}$ forms an orthonormal basis of $T_q\mathcal{M}$. In addition, the orthonormal frame defined in this way is unique and depends smoothly on $q$. Therefore, we obtain on $W$ a system of normal coordinates at each $q\in W$, which parameterizes $x\in\mathcal{E}_q(B_{\delta}(0))$ by
 \begin{equation}\label{eq:local}
    x=\mathcal{E}_q\bigg(\sum_{i=1}^du_ie^q_i\bigg)=\phi^q(u_1,\ldots,u_d),\ u=(u_1,\ldots,u_d)\in B_{\delta}(0).
 \end{equation}
 Such coordinates are called $q$-normal coordinates. The basis of $T_q\mathcal{M}$
associated with this coordinate chart $(B_{\delta}(0),\phi^q)$ is given by
\[
\frac{\partial}{\partial u_i}(q)(f)=\frac{d(f\circ\mathcal{E}_q(te^q_i))}{dt}\bigg|_{t=0}=\frac{d(f\circ\gamma(t,q,e_i^q))}{dt}\bigg|_{t=0}=e_i^q(f), \ i=1,\ldots,d.
\]
Therefore $\{\frac{\partial}{\partial u_i}(q)=e_i^q:1\leq i\leq d\}$ forms an orthonormal basis on $T_q\mathcal{M}$. By Proposition \ref{prop:1}, for each $x\in\mathcal{E}_q(B_{\delta}(0))$, there exists a minimizing geodesic $\gamma(t,q,v),0\leq t\leq 1$, such that $\gamma(0,q,v)=q,\gamma'(0,q,v)=v$ and $\gamma(1,q,v)=x$, where $v=\mathcal{E}_q^{-1}(x)=\sum_{i=1}^du_ie^q_i\in T_q\mathcal{M}$. Hence $d_{\mathcal{M}}(q,x)=\int_0^1|\gamma'(t,q,v)|dt=|v|=||u||$, i.e.
\begin{equation}\label{eq:1b}
d_{\mathcal{M}}\bigg(q,\mathcal{E}_q\bigg(\sum_{i=1}^du_ie^q_i\bigg)\bigg)=||u||,\quad\forall u\in B_{\delta_p}(0),
\end{equation}
where $||\cdot||$ is the Euclidean norm on $\bbR^d$.
 The components $g_{ij}^q(u)$ of the Riemannian metric in $q$-normal coordinates satisfy $g_{ij}^q(0)=g_q(e_i^q,e_j^q)=\delta_{ij}$. The following results \citep[Proposition 2.2]{Gine2005} provide local expansions for the Riemannian metric $g_{ij}^q(u)$, the Jacobian $\sqrt{\det(g^{\phi}_{ij}(u))}$ and the distance $d_{\mathcal{M}}(q,\sum_{i=1}^du_ie^q_i)$ in a neighborhood of $p$.

\begin{proposition}\label{prop:2}
Let $\mathcal{M}$ be a submanifold of $\bbR^D$ which is isometrically embedded.
Given a point $p\in\mathcal{M}$, let $W$ and $\delta$ be as in Proposition  \ref{prop:1}, and consider for each $q\in W$ the $q$-normal coordinates defined above. Suppose that $x=\sum_{i=1}^du_ie^q_i\in\mathcal{E}_q(B_{\delta}(0))$. Then:
\begin{enumerate}
  \item[(1)] The components $g_{ij}^q(u)$ of the metric tensor in $q$-normal coordinates admit the following expansion, uniformly in $q\in W$ and $x\in\mathcal{E}_q(B_{\delta}(0))$:
      \begin{equation}\label{eq:2b}
       g_{ij}^q(u_1,\ldots,u_d)=\delta_{ij}-\frac{1}{3}\sum_{r,s=1}^dR^q_{irsj}(0)u_ru_s+O(d^3_{\mathcal{M}}(q,x)),
      \end{equation}
      where $R^q_{irsj}(0)$ are the components of the curvature tensor at $q$ in $q$-normal coordinates.
  \item[(2)] The Jacobian $\sqrt{\det(g_{ij}^q)}(u)$ in $q$-normal coordinates has the following expansion, uniformly in $q\in W$ and $x\in\mathcal{E}_q(B_{\delta}(0))$:
      \begin{equation}\label{eq:3b}
       \sqrt{\det(g_{ij}^q)}(u_1,\ldots,u_d)=1-\frac{1}{6}\sum_{r,s=1}^d\text{Ric}^q_{rs}(0)u_ru_s+O(d^3_{\mathcal{M}}(q,x)),
      \end{equation}
      where $\text{Ric}^q_{rs}(0)$ are the components of the Ricci tensor at $q$ in $q$-normal coordinates.
  \item[(3)] There exists $C_p<\infty$ such that
  \begin{equation}\label{eq:4b}
    0\leq d^2_{\mathcal{M}}(q,x)-||q-x||^2\leq C_pd^4_{\mathcal{M}}(q,x)
  \end{equation}
  holds uniformly in $q\in W$ and $x\in\mathcal{E}_q(B_{\delta}(0))$.
\end{enumerate}
\end{proposition}

Note that in Proposition \ref{prop:2}, (3) only provides a comparison of geodesic distance and Euclidean distance in local neighborhoods. Under a stronger compactness assumption on $\mathcal{M}$, the following lemma offers a global comparison of these two distances.

\begin{lemma}\label{le:8}
Let $\mathcal{M}$ be a connected compact submanifold of $\bbR^D$ with a Riemannian metric $g$ that is not necessarily induced from the Euclidean metric. Then there exist positive constants $C_1$ and $C_2$ dependent on $g$, such that
\begin{equation}\label{eq:5b}
C_1||x-y||\leq d_{\mathcal{M}}(x,y)\leq C_2||x-y||,\quad \forall x,y\in\mathcal{M},
\end{equation}
where $||\cdot||$ is the Euclidean distance in $\bbR^D$. Moreover, if $\mathcal{M}$ is further assumed to be isometrically embedded, i.e. $g$ is induced from the Euclidean metric of $\bbR^D$, then $C_1$ could be chosen to be one and $C_2\geq 1$.
\end{lemma}

\begin{proof}
We only prove the first half of the inequality since the second half follows by a similar argument and is omitted here. Assume in the contrary that for some sequence $\{M_k\}$ satisfying $M_k\to\infty$ as $k\to\infty$, there exists $(x_k,y_k)$ such that $||x_k-y_k||\geq M_k d_{\mathcal{M}}(x_k,y_k)$. Let $\Phi:\mathcal{M}\rightarrow\bbR^D$ be the embedding. Since $\mathcal{M}$ is compact, $\{x_k\}$ and $\{y_k\}$ have convergent subsequences, whose notations are abused as $\{x_k\}$ and $\{y_k\}$ for simplicity. Denote the limits of these two sequences as $x_0$ and $y_0$. By the compactness of $\mathcal{M}$ and continuity of $\Phi$, we know that $\Phi(\mathcal{M})$ is also compact and therefore $d_{\mathcal{M}}(x_k,y_k)\rightarrow 0$, as $k\rightarrow\infty$. This implies that $x_0=y_0=p$.

For each $j\in\{1,\ldots,p\}$, the $j$th component $\Phi_j:\mathcal{M}\rightarrow\bbR$ of $\Phi$ is differentiable. Let $\delta_{p}$ and $W_p$ be the $\delta$ and $W$ specified in Proposition \ref{prop:1}. Define $f(q,v)=\Phi\big(\pi_2(F(q,v))\big)=\Phi(\mathcal{E}_p(v))$, where $\pi_2$ is the projection of $\mathcal{M}\times\mathcal{M}$ on to its second component. By Proposition \ref{prop:1}, $f$ is differentiable on the compact set $\bar{W}_{\delta_p}$ and therefore for each $l\in\{1,\ldots,d\}$, $\frac{\partial f}{\partial v_l}$ is uniformly bounded on $\bar{W}_{\delta_p}$. This implies that for some constant $C>0$, $||x-y||=||f(y,\mathcal{E}_y^{-1}(x))-f(y,\mathcal{E}_y^{-1}(y))||\leq C ||\mathcal{E}_y^{-1}(x)-\mathcal{E}_y^{-1}(y)||=Cd_{\mathcal{M}}(x,y)$ for all $x,y\in W_p$ with $d_{\mathcal{M}}(x,y)\leq\delta_p$. Since $x_k\rightarrow p$ and $y_k\rightarrow p$, there exists an integer $k_0$ such that for all $k>k_0$, $x_k$, $y_k\in W_p$ and $d_{\mathcal{M}}(x_k,y_k)\leq\delta_p$. Therefore $||x_k-y_k||\leq Cd_{\mathcal{M}}(x_k,y_k)$, which contradicts our assumption that $||x_k-y_k||\geq M_k d_{\mathcal{M}}(x_k,y_k)$ for all $k$.

Consider the case when $\Phi$ is an isometric embedding. For any points $x,y\in \mathcal{M}$, we can cover the compact geodesic path $l_{x,y}$ from $x$ to $y$ by $\{W_{p_i}:i=1,\ldots,n\}$ associated with a finite number of points $\{p_1,\ldots,p_n\}\subset\mathcal{M}$. Therefore we can divide $l_{x,y}$ into $\bigcup_{s=1}^n l(x_{s-1},x_s)$ such that $x_0=x$, $x_{n}=y$, and each segment $l(x_{s-1},x_s)$ lies in one of the $W_{p_i}$'s. By Proposition \ref{prop:2}\,(3), for each $s\in\{1,\ldots,n\}$, $d_{\mathcal{M}}(x_{s-1},x_s)\geq ||x_{s-1}-x_s||$. Therefore,
\begin{align*}
    d_{\mathcal{M}}(x,y)=\sum_{s=1}^nd_{\mathcal{M}}(x_{s-1},x_s)\geq \sum_{s=1}^n||x_{s-1}-x_s||\geq ||x-y||,
\end{align*}
where the last step follows from the triangle inequality.
\end{proof}

The above lemma also implies that geodesic distances induced by different Riemannian metrics on $\mathcal{M}$ are equivalent to each other.

Fix $p\in\mathcal{M}$ and let $W$ and $\delta>0$ be specified as in Proposition \ref{prop:1}. Since $\mathcal{M}$ is a submanifold of $\bbR^D$, for any point $q\in \mathcal{M}$, the exponential map $\mathcal{E}_q:B_{\delta}(0)\rightarrow\mathcal{M}\subset R^D$ is a differentiable function between two subsets of Euclidean spaces. Here, we can choose any orthonormal basis of $T_q\mathcal{M}$ since the representations of $\mathcal{E}_q$ under different orthonormal bases are the same up to $d\times d$ rotation matrices. Under the compactness assumption on $\mathcal{M}$, the following lemma ensures the existence of a bound on the partial derivatives of $\mathcal{E}_q$'s components $\{\mathcal{E}_{q,i}:i=1,\ldots,D\}$ uniformly for all $q$ in the $\delta$ neighborhood of $p$:

\begin{lemma}\label{le:bd}
Let $\mathcal{M}$ be a connected $C^{\gamma}$ compact submanifold of $\bbR^D$ with $\gamma$ being $\infty$ or any integer greater than two. Let $k$ be an integer such that $k\leq\gamma$. Then:
\begin{enumerate}\itemsep=-1pt
  \item There exists a universal positive number $\delta_0$, such that for every $p\in\mathcal{M}$, proposition \ref{prop:1} is satisfied with some $\delta>\delta_0$ and $W_p$;
  \item With this $\delta_0$, for any $p\in\mathcal{M}$, mixed partial derivatives with order less than or equal to $k$ of each component of $\mathcal{E}_p$ are bounded in $B_{\delta_0}(0)\in T_p\mathcal{M}$ by a universal constant $C>0$.
\end{enumerate}
\end{lemma}

\begin{proof}
Note that $\mathcal{M}=\bigcup_{p\in\mathcal{M}}W(p,\delta_p)$, where $\delta_p$ and $W(p,\delta_p)$ are the corresponding $p$ dependent $\delta$ and open neighborhood $W$ in proposition \ref{prop:1}. By the compactness of $\mathcal{M}$, we can choose a finite covering $\{W(p_1,\delta_{p_1}),\ldots,W(p_n,\delta_{p_n})\}$. Let $\delta_0=\min\{\delta_{p_1},\ldots,\delta_{p_n}\}$. Then the first condition is satisfied with this $\delta_0$ since for any $p\in\mathcal{M}$, $W_p$ could be chosen as any $W(p_j,\delta_{p_j})$ that contains $p$.

Next we prove the second condition. For each $j$, we can define $q$-normal coordinates on $W(p_j,\delta_{p_j})$ as before such that the exponential map at each point $q\in W(p_j,\delta_{p_j})$ can be parameterized as \eqref{eq:local}.
Define $F_j:W(p_j,\delta_{p_j})\times B_{\delta_{p_j}}(0)\rightarrow \bbR^D$ by $F_j(q,u)=\mathcal{E}_q(\sum_{i=1}^du_ie^q_i)=\phi^q(u)$. Then any order $k$ mixed partial derivative $\frac{\partial^k \phi^q_j}{\partial{u_{i1}}\cdots\partial{u_{ik}}}(u)$ of $F_j(q,u)$ with respect to $u$ is continuous on the compact set $W(p_j,\delta_{p_j})\times B_{\delta_{p_j}}(0)$. Therefore these partial derivatives are bounded uniformly in $q\in W(p_j,\delta_{p_j})$ and $u\in B_{\delta_{p_j}}(0)$. Since $\mathcal{M}$ is covered by a finite number of sets $\{W(p_1,\delta_{p_1}),\ldots,W(p_n,\delta_{p_n})\}$, the second conclusion is also true.
\end{proof}

By lemma \ref{le:bd}, when a compact submanifold $\mathcal{M}$ has smoothness level greater than or equal to $k$, we can approximate the exponential map $\mathcal{E}_{p}:B_{\delta_0}(0)\subset \bbR^d\rightarrow \bbR^D$ at any point $p\in\mathcal{M}$ by a local Taylor polynomial of order $k$ with error bound $C\delta_0^k$, where $C$ is a universal constant that only depends on $k$ and $\mathcal{M}$.

\section{Proofs of the main results}
\subsection{Proof of Theorem 2.1}
Define centered and decentered concentration functions of the process $W^a=(W_{ax}:x\in\mathcal{M})$ by
\begin{eqnarray*}
&&\phi_0^a(\epsilon)=-\log P(|W^a|_{\infty}\leq \epsilon),\\
&&\phi_{f_0}^a(\epsilon)=\inf_{h\in\tilde{\mathbb{H}}^{a}:|h-f_0|_{\infty}\leq \epsilon}||h||^2_{\tilde{\mathbb{H}}^{a}}-\log P(|W^a|_{\infty}\leq \epsilon),
\end{eqnarray*}
where $|h|_{\infty}=\sup_{x\in \mathcal{M}}|f(x)|$ is the sup norm on the manifold $\mathcal{M}$.
Then $P(|W^a|_{\infty}\leq\epsilon)=\exp(-\phi_0^a(\epsilon))$ by definition. Moreover, by the results in \cite{Kuelbs1994},
\begin{equation}\label{eq:2}
P(||W^a-f_0||_{\infty}\leq2\epsilon)\geq e^{-\phi_{f_0}^a(\epsilon)}.
\end{equation}
Suppose that $f_0\in C^s(\mathcal{M})$ for some $s\leq\min\{2,\gamma-1\}$. By Lemma 3.5 and Lemma 3.2, for $a>a_0$ and $\epsilon>C\max\{a^{-(\gamma-1)},a^{-s}\}=Ca^{-s}$,
\[
\phi_{f_0}^s(\epsilon)\leq Da^d+C_4a^d\bigg(\log\frac{a}{\epsilon}\bigg)^{1+d}\leq K_1a^d\bigg(\log\frac{a}{\epsilon}\bigg)^{1+d}.
\]
Since $A^d$ has a Gamma prior, there exists $p,C_1,C_2>0$, such that $ C_1 a^p\exp(-D_2a^d)\leq g(a)\leq C_2 a^p\exp(-D_2a^d)$. Therefore by equation (\ref{eq:2}),
\begin{align*}
    P(||W^A-f_0||_{\infty}\leq 2\epsilon)& \geq P(||W^A-f_0||_{\infty}\leq 2\epsilon, A\in[(C/\epsilon)^{1/s}, 2(C/\epsilon)^{1/s}])\\
    &\geq \int_{(C/\epsilon)^{1/s}}^{2(C/\epsilon)^{1/s}}e^{-\phi_{f_0}^s(\epsilon)}g(a)da\\
    &\geq C_1e^{-K_2(1/\epsilon)^{d/s}(\log(1/\epsilon))^{1+d}}\bigg(\frac{C}{\epsilon}\bigg)^{p/s}\bigg(\frac{C}{\epsilon}\bigg)^{1/s}.
\end{align*}
Therefore,
\[
P(||W^A-f_0||_{\infty}\leq \epsilon_n)\geq \exp(-n\epsilon_n^2),
 \]
 for $\epsilon_n$ a large multiple of $n^{-s/(2s+d)}(\log n)^{\kappa_1}$ with $\kappa_1=(1+d)/(2+d/s)$ and sufficiently large $n$.

Similar to the proof of Theorem 3.1 of \cite{Van2009}, by Lemma 3.6,
\[
B_{M,r,\delta,\epsilon}=\bigg(M\sqrt{\frac{r}{\delta}}\tilde{\mathbb{H}}^{r}_1+\epsilon\mathbb{B}_1\bigg)\cup
\bigg(\bigcup_{a<\delta}(M\tilde{\mathbb{H}}^{a}_1)+\epsilon\mathbb{B}_1\bigg),
\]
with $\mathbb{B}_1$ the unit ball of $C(\mathcal{M})$, contains the set $M\tilde{\mathbb{H}}^{a}_1+\epsilon\mathbb{B}_1$ for any $a\in[\delta,r]$. Furthermore, if
\begin{equation}\label{eq:4}
M\geq 4\sqrt{\phi_0^r(\epsilon)}\ \ \text{ and }\ \ e^{-\phi_0^r(\epsilon)}<1/4,
\end{equation}
then
\begin{equation}\label{eq:3}
P(W^A\notin B)\leq \frac{2C_2r^{p-d+1}e^{-D_2r^d}}{D_2d}+e^{-M^2/8}.
\end{equation}
By Lemma 4.5, equation (\ref{eq:4}) is satisfied if
\[
M^2\geq 16C_4r^d(\log(r/\epsilon))^{1+d},\quad r>1,\quad \epsilon<\epsilon_1,
\]
for some fixed $\epsilon_1>0$.
Therefore
\[
P(W^A\notin B)\leq \exp(-C_0n\epsilon^2_n),
\]
for $r$ and $M$ satisfying
\begin{equation}\label{eq:6}
r^d=\frac{2C_0}{D_2}n\epsilon_n^2,\quad M^2=\max\{8C_0,16C_4\}n\epsilon_n^2(\log(r/\epsilon_n))^{1+d}.
\end{equation}
Denote the solution of the above equation as $r_n$ and $M_n$.

By Lemma 3.4, for $M\sqrt{r/\delta}>2\epsilon$ and $r>a_0$,
\begin{align*}
\log N\bigg(2\epsilon,M\sqrt{\frac{r}{\delta}}\tilde{\mathbb{H}}^{r}_1+\epsilon\tilde{\mathbb{B}}_1,||\cdot||_{\infty}\bigg)&\leq
\log N\bigg(\epsilon,M\sqrt{\frac{r}{\delta}}\tilde{\mathbb{H}}^{r}_1,||\cdot||_{\infty}\bigg)\\
&\leq Kr^d\bigg(\log\bigg(\frac{M\sqrt{r/\delta}}{\epsilon}\bigg)\bigg)^{1+d}.
\end{align*}
By Lemma 3.7, every element of $M\tilde{\mathbb{H}}^{a}_1$ for $a<\delta$ is uniformly at most $\delta\sqrt{D}\tau M$ distant from a constant function for a constant in the interval $[-M,M]$. Therefore for $\epsilon>\delta\sqrt{D}\tau M$,
\[
\log N\bigg(3\epsilon,\bigcup_{a<\delta}(M\tilde{\mathbb{H}}^{a}_1)+\epsilon\tilde{\mathbb{B}}_1,||\cdot||_{\infty}\bigg)
\leq N(\epsilon,[-M,M],|\cdot|)\leq\frac{2M}{\epsilon}.
\]
With $\delta=\epsilon/(2\sqrt{D}\tau M)$, combining the above displays, for $B=B_{M,r,\delta,\epsilon}$ with
\[
{M}\geq\epsilon,\ {M^{3/2}\sqrt{2\tau r}D^{1/4}}\geq2{\epsilon^{3/2}},\  r>a_0,
\]
which is satisfied when $r=r_n$ and $M=M_n$,
we have
\begin{equation}\label{eq:5}
\begin{aligned}
\log N\big(3\epsilon,B,||\cdot||_{\infty}\big)&\leq Kr^d\bigg(\log\bigg(\frac{M^{3/2}\sqrt{2\tau r}D^{1/4}}{\epsilon^{3/2}}\bigg)\bigg)^{1+d}
+\log\frac{2M}{\epsilon}.
\end{aligned}
\end{equation}
Therefore, for $r=r_n$, $M=M_n$ and $B_n=B_{M_n,r_n,\delta_n,\epsilon_n}$,
\[
\log N\big(3\bar{\epsilon}_n,B_n,||\cdot||_{\infty}\big)\leq n\bar{\epsilon}_n^2,
\]
for $\bar{\epsilon}_n$ a large multiple of $\epsilon_n(\log n)^{\kappa_2}$ with $\kappa_2=(1+d)/2$.

\subsection{Proof of Corollary 2.2}
Under $d'$, the prior concentration inequality becomes:
\begin{align}
    P(||W^A-f_0||_{\infty}\leq 2\epsilon)& \geq P(||W^A-f_0||_{\infty}\leq 2\epsilon, A\in[(C/\epsilon)^{1/s}, 2(C/\epsilon)^{1/s}])\nonumber\\
    &\geq \int_{(C/\epsilon)^{1/s}}^{2(C/\epsilon)^{1/s}}e^{-\phi_{f_0}^s(\epsilon)}g(a)da\nonumber\\
    &\geq C_1e^{-K_2(1/\epsilon)^{d\vee d'/s}(\log(1/\epsilon))^{1+d}}\bigg(\frac{C}{\epsilon}\bigg)^{p/s}\bigg(\frac{C}{\epsilon}\bigg)^{1/s}.
    \label{eq:pc2}
\end{align}

The complementary probability becomes:
\begin{align}
    P(W^A\notin B)\leq \frac{2C_2r^{p-d'+1}e^{-D_2r^{d'}}}{D_2}+e^{-M^2/8},\label{eq:cp2}
\end{align}
with $M^2\geq 16C_4r^d(\log(r/\epsilon))^{1+d}$, $r>1$ and $\epsilon<\epsilon_1$, where $\epsilon_1>0$ is a fixed constant.

An upper bound for the covering entropy is given by
\begin{equation}\label{eq:logn}
\begin{aligned}
\log N\big(3\epsilon,B,||\cdot||_{\infty}\big)&\leq Kr^d\bigg(\log\bigg(\frac{M^{3/2}\sqrt{2\tau r}D^{1/4}}{\epsilon^{3/2}}\bigg)\bigg)^{1+d}
+\log\frac{2M}{\epsilon}.
\end{aligned}
\end{equation}

\emph{1. $d'>d$:} With $\epsilon_n$ a multiple of $n^{-s/(2s+d')}(\log n)^{\kappa_1}$ with $\kappa_1=(1+d)/(2+d'/s)$, $\bar{\epsilon}_n<\epsilon_n$, $$r^{d'}=\frac{2C_0}{D_2}n\epsilon_n^2,\text{ and } M^2=\max\{8C_0,16C_4\}n\epsilon_n^2(\log(r/\epsilon_n))^{1+d},$$ inequalities \eqref{eq:pc2}, \eqref{eq:cp2} and \eqref{eq:logn} becomes \eqref{eq:0c}.
Therefore we arrive at the conclusion that under $d'>d$, the posterior contraction rate will be at least a multiple of $n^{-s/(2s+d')}(\log n)^{\kappa}$ with $\kappa=(1+d)/(2+d'/s)$.

\emph{2. $\frac{d^2}{2s+d}<d'<d$:} With $\epsilon_n$ a multiple of $n^{-s/(2s+d)}(\log n)^{\kappa_1}$ with $\kappa_1=(1+d)/(2+d/s)$, $\bar{\epsilon}_n$ a multiple of $n^{d/(2d')-1}\epsilon_n^{d/d'}(\log n)^{(d+1)/2}=n^{-\frac{(2s+d)d'-d^2}{2(2s+d)d'}}(\log n)^{\kappa_2}$ with $\kappa_2=(d+d^2)/(2d'+dd'/s)+(1+d)/2$,
 $$r^{d'}=\frac{2C_0}{D_2}n\epsilon_n^2,\text{ and } M^2=\max\{8C_0,16C_4\}n\epsilon_n^2(\log(r/\epsilon_n))^{1+d},$$ inequalities \eqref{eq:pc2}, \eqref{eq:cp2} and \eqref{eq:logn} becomes \eqref{eq:0c}.
Therefore we arrive at the conclusion that under $d'<d$, the posterior contraction rate will be at least a multiple of $n^{-\frac{(2s+d)d'-d^2}{2(2s+d)d'}}(\log n)^{\kappa}$ with $\kappa=(d+d^2)/(2d'+dd'/s)+(1+d)/2$. To make this rate meaningful, we need $(2s+d)d'-d^2>0$, i.e. $d'>d^2/(2s+d)$.

\bibliography{draft}
\end{document}